\documentclass{amsart}
\usepackage[utf8]{inputenc}

\title{Group bundles and group connections}

\author[D. Bl\'azquez-Sanz]{David Bl\'azquez-Sanz}
\author[S.S. Gordon]{Sedney Su\'arez Gordon}
\author[C. Mar\'in]{Carlos Alberto Mar\'in Arango}

\address{Instituto de Matem\'aticas \hfill\break\indent  Universidad de Antioquia \hfill\break\indent Medell\'in, Colombia}
\email{shahadgordon92@gmail.com}
\email{calberto.marin@udea.edu.co}
\address{Universidad Nacional de Colombia - Sede Medell\'in \hfill\break\indent  Facultad de Ciencias 
\hfill\break\indent Escuela de Matem\'aticas \hfill\break\indent  Medell\'in, Colombia}
\email{dblazquezs@unal.edu.co}



\date{July 2021}

\usepackage[english]{babel}
\usepackage{times}
\usepackage{dsfont}

\usepackage{amsmath,amssymb,hyperref,color,amsthm,verbatim}   
\usepackage[mathscr]{euscript}
\usepackage{pb-diagram}
\usepackage[all]{xy}
\usepackage{cleveref}
\usepackage{multirow}
\usepackage{multicol}
\begin{document}

\makeatletter
\newenvironment{dem}{Proof:}{\qed}

\makeatother

\numberwithin{equation}{section}
\theoremstyle{plain}\newtheorem{theorem}{Theorem}[section]
\theoremstyle{plain}\newtheorem{proposition}[theorem]{Proposition}
\theoremstyle{plain}\newtheorem{lemma}[theorem]{Lemma}
\theoremstyle{plain}\newtheorem{corollary}[theorem]{Corollary}

\theoremstyle{definition}\newtheorem{definition}[theorem]{Definition}
\theoremstyle{remark}\newtheorem{remark}[theorem]{Remark}

\theoremstyle{definition}\newtheorem{example}[theorem]{Example}

\newcommand{\fn}[5]{
            \begin{array}{cccl}
                #1: & #2 & \longrightarrow &  {#3}\\
                    & #4 & \longmapsto & {#1}(#4)={#5}
            \end{array}}
            
\newcommand{\na}[3]{\left( \mathcal{L}_{#1}\nabla \right)_{#2}{#3}=0} 

\newcommand{\dx}{\partial_x}
\newcommand{\dy}{\partial_y}
\newcommand{\dxi}{\partial_{x_i}}
\newcommand{\dxj}{\partial_{x_j}}
\newcommand{\dxk}{\partial_{x_k}}
\newcommand{\dxu}{\partial_{x_1}}
\newcommand{\dxd}{\partial_{x_2}}
\newcommand{\dt}{\partial_{\theta}}
\newcommand{\dr}{\partial_{r}}

\def\homo{\mathrm{Hom}}
\def\L{\mathcal{L}}
\def\r{\mathbb{R}}
\def\n{\mathbb{N}}
\def\q{\mathbb{Q}}
\def\z{\mathbb{Z}}
\def\so{\text{SO}}
\def\o{\text{O}}
\def\s{\mathbb{S}}
\def\c{\mathbb{C}}
\def\g{\mathfrak{g}} 
\def\cp{\mathbb{CP}} 
\def\fibP{\pi : P\to M}
\def\fibV{\pi : E\to M}
\def\d{\mathrm{d}}
\def\h{\mathcal{H}}
\def\gl{{\rm gl}(n)}
\def\X{\mathfrak{X}}
\def\lin{\rm Lin}
\def\gij{\Gamma_{ij}^l}
\def\gu{\Gamma_{11}}
\def\gd{\Gamma_{22}}
\def\gud{\Gamma_{12}}
\def\gdu{\Gamma_{21}}

\maketitle

\begin{abstract}
    We consider smooth families of Lie groups (group bundles) and connections that are compatible with the group operation. We characterize the space of group connections on a group bundle as an affine space modeled over the vector space of $1$-forms with values cocycles in the Lie algebra bundle of the aforementioned group bundle. We show that group connections satisfy the Ambrose-Singer theorem and that group bundles can be seen as a particular case of associated bundles realizing group connections as associated connections. We give a construction of the Moduli space of group connections with fixed base and fiber, as an space of representations of the fundamental group of the base.  
\end{abstract}


\section*{Introduction}

We study smooth families of Lie groups and the differential equations that are compatible with the group operation. As it has been pointed by B. Malgrange, who in \cite{malgrange2010differential} developed the theory of group connections in the complex algebraic case, this question is implicit in the work of E. Cartan on ``infinite dimensional Lie groups''. It also constitutes a geometric counterpart of the notion of differential algebraic groups introduced by E. Kolchin and developed by P. Cassidy and A. Buium among others \cite{cassidy1972differential, kolchin1985differential, buium2006differential}.

The purpose of this work is to outline the fundamentals of the theory of smooth group bundles and group connections, and how this theory is connected with the well known theory of transitive Lie groupoids and principal bundles. In particular, we show that group bundles can always be seen as associated bundles and group connections are always recovered as associated connections induced by principal connections in principal bundles.
In order to do so, we develop some general tools in the theory of connections. An implicit assumption that we use throughout the article is that, for any Lie group $G$ considered, the group of components $G/G^0$ is finitely generated. This gives to the group ${\rm Aut}_{gr}(G)$ of Lie group automorphisms \cite[Theorem 2]{hochschild1952automorphism} a canonical structure of Lie group.

\subsection*{Contents and results}

We set our notations and conventions about connections in Section \ref{s_EC}. Our treatment of connections emphasizes a triple approach: connections as distributions of vector fields, covariant derivative operators and sections of the jet bundle. Connections allowing arbitrary ``analytic continuation'' of horizontal piece-wise smooth paths are termed \emph{Ehresmann connections}. They have holonomy groups of diffeomorphisms. We check that they satisfy some weak form of the Ambrose-Singer theorem (Theorem \ref{th:w_AS}).

Section \ref{s_GB} is devoted to the general definitions around the notions of group bundle and group connection. We see that group bundles admitting group connections are locally trivial (Theorem \ref{th:LT}). Then we explore the space of group connection, finding that it is an affine space modeled over the vector space of $1$-forms in the base space with values in the space of $1$-cocycles of the group bundle in its own Lie algebra (Theorem \ref{th:GC_AS}).

In Section \ref{s_FB} we set up what we need about principal and associated bundles. In particular, we explore the relation between principal and associated connections, and their curvatures. Next, we start Section \ref{s_TSLG} by exploring the well known correspondence between principal bundles and transitive Lie groupoids. We see that such correspondence extends to relate associated bundles with representations of transitive Lie groupoids, and associated connections with connections compatible with such groupoid representation (Theorem \ref{th:GCC}). We have also realized that locally trivial group bundles can always be recovered as associated bundles to certain principal bundles (Theorem \ref{th:group_associated}). It immediately follows that group connections can always be recovered as associated connections to principal connections (Corollary \ref{cor:G_vs_AutG}). This also implies that group connections satisfy the Ambrose-Singer theorem (Corollary \ref{cor:AS_GC}.). Finally, in section \ref{s_ME} we explicitly construct the Moduli space of group connections on a fixed base space and fiber. This construction is completely analogous to that of the Moduli space of principal connections in principal bundles with fixed base space and structure group.

\section{Ehresmann connections}\label{s_EC}

Let $E$, $M$ be differentiable manifolds and let $\pi\colon E\to M$ be a smooth
submersion. Given $e\in  E$, the space $\mathrm{Ker}\big(\mathrm{d}_e \pi\big)$
is called the {\em vertical subspace\/} of $T_e E$ at the point $e$ with respect
to $\pi$; and it will be denoted by $\mathrm{Ver}_e(E)$.
A subspace $H$ of $T_e E$ is called
{\em horizontal\/} with respect to $\pi$ if it is a complement of $\mathrm{Ver}_e(E)$
in $T_e E$, that is, if:
\[T_e E=H\oplus\mathrm{Ver}_e(E).\]
Clearly a subspace $H$ of $T_e E$ is horizontal with respect to $\pi$ if and only
if the restriction of $\mathrm{d}_e\pi$ to $H$ is an isomorphism onto $T_{\pi(e)}M$.

A distribution $\mathcal H$ on the manifold $E$ is called {\em horizontal\/} with respect to $\pi$ if $\mathcal H_e$ is a horizontal subspace of $T_e E$ for
every $e \in E$. A smooth horizontal distribution on $E$ will be called an \emph{connection} on $E$ (with respect to $\pi$).

For all $x\in M$, $\pi^{-1}(x):=E_x$ is a smooth submanifold of $E$ and for every $e\in E_x$ we have:
\[\mathrm{Ver}_e(E)=T_e\big(E_x\big).\]
Therefore
\[\mathrm{Ver}(E):=\bigcup_{e\in E}\mathrm{Ver}_e(E)\subset TE\] is a smooth distribution on $E$, called the {\em vertical bundle\/} of $E$ determined by $\pi$.

Given a horizontal distribution $\mathcal{H}$ on $E$; then
\begin{equation}\label{eq:descHorplusVer}
T_e E = \mathcal{H}_e \oplus \mathrm{Ver}_p(E),
\end{equation}
for all $e\in E$. In this case we will denote by $\mathfrak{p}_{\mathrm{ver}}:T E\to\mathrm{Ver}(E)$, $\mathfrak{p}_{\mathrm{\mathcal{H}}}:TE\to\mathcal{H}$ the maps whose restrictions to $T_e E$ are equal to
the projection onto the first and the second coordinate respectively, corresponding to the direct sum decomposition \ref{eq:descHorplusVer}. 

\begin{definition}
Let $E, M$ be smooth manifolds and let $\pi:E\to M$ be a smooth submersion. Let $\mathcal{H}$ be an connection on $E$. If $\epsilon:U\to E$ is a smooth local section of $\pi$ then, given $x\in U$, $\vec v\in T_xM$, {\em{the covariant derivarive\/}} of $\epsilon$ at the point $x$ in the direction of $v$ with respecto to $\mathcal{H}$ is denoted and defined by
\begin{equation}\label{eq:covariantDeriv}
\nabla_{\vec v}\epsilon = \mathfrak{p}_{\mathrm{ver}}\left( \mathrm{d}_x \epsilon (\vec v) \right) \in \mathrm{Ver}_{\epsilon(x)}E;
\end{equation}
$\nabla$ is called {\em{the covariant derivative operator\/}} associated to the Ehreshmann connection $\mathcal{H}$. Given $x\in U$, if $\nabla_{\vec v} \epsilon =0$, for all $\vec v\in T_xM$, the local section $\epsilon$ is said to be {\em{horizontal at $x$\/}} with respect to $\mathcal{H}$; when this is the case for all $x\in U$, we say simply that $\epsilon$ {\em{is horizontal\/}} with respect to $\mathcal{H}$.
\end{definition}


\subsection{Connections as sections of the jet bundle}

We say that two local sections $\epsilon_1$, $\epsilon_2$ of $\pi\colon E\to M$ have contact of order $\geq 1$ at a point $x\in M$ in its common domain of definition if $\epsilon_1(x) = \epsilon_2(x)$ and $\mathrm{d}_{x}\epsilon_1 = \mathrm{d}_{x}\epsilon_2$. An equivalence class of contact of order $\geq 1$ of local sections of $\pi$ at a point of $M$ is called a $1$-jet of local section. We write $j^1_x\epsilon$ for the $1$-jet at $x$ of the local section $\epsilon$. The space of $1$-jets of sections of $\pi$ that we denote by $J^1E$ admits a canonical smooth structure such that 
there is a natural closed immersion,
$$J^1E \hookrightarrow \pi^*(T^*M)\otimes TE, \quad j^1_x\epsilon \mapsto \mathrm{d}_x\epsilon.$$
It is clear that for any two local sections $\epsilon_1$, $\epsilon_2$ with $\epsilon_1(x) = \epsilon_2(x) = p$ the difference $\mathrm{d}_x\epsilon_1 - \mathrm{d}_x\epsilon$ takes values in ${\rm Ver}_{p}(E)$. It follows that $\pi_{1,0}\colon J^1E \to E$ is an affine subbundle of $\pi^*(T^*M)\otimes TE$ modeled over the vector bundle $\pi^*(T^*M)\otimes {\rm Ver}(E)$. From now on we will make use of the fact that any element $\varphi\in J^1E$ can be seen as a linear map $\varphi\colon T_{\pi_{1,-1}(\varphi)}M \to T_{\pi_{0,1}(\varphi)}E$ where $\pi_{1,-1}$ stands for $\pi\circ\pi_{1,0}$. 

As this linear map $\varphi\in J^1E$ is a section of $\mathrm{d}_{\pi_{1,0}(\varphi)}\pi$ it follows that it is completely determined by its image $\varphi(T_{\pi_{1,-1}(\varphi)})\subset T_{\pi_{1,0}(\varphi)}E$ which is an horizontal subspace of $TE$. Therefore we have also a canonical inclusion,
$$J^1E \hookrightarrow {\rm Grass}(TE,m), \quad \varphi = j^1_x\epsilon \mapsto \varphi(T_{x}M) = \mathrm{d}_x\epsilon(T_xM),$$
where $m=\dim(M)$ and ${\rm Grass}(TE,m)$ stands for the grassmanian bundle of $m$ dimensional subspaces of $TE$. Such inclussion identifies $J^1E$ with the open dense subset of horizontal subspaces of $TE$.  

Therefore, there is a natural bijective correspondence between sections of $\pi_{1,0}$ and distributions of horizontal spaces (connections) for $\pi$. A section $\sigma$ of $\pi_{1,0}$ defines an Ehreshmann connection $\mathcal H$ by setting $\mathcal H_p  := \sigma(p)(T_{\pi(p)}M)$, and reciprocally an Ehreshmann connection $\mathcal H$ defines defines a section $\sigma$ of $\pi_{1,0}$ by setting $\sigma(p) = (\mathrm{d}_p\pi|_{\mathcal H_p})^{-1}$.

\subsection{Horizontal lift and curvature}

\begin{definition}
Let $\pi\colon E\to M$ be a smooth submersion endowed with a connection $\mathcal H$. Given a smooth vector field $X $ on $M$, the {\em horizontal lift\/} of $X$ is defined as the unique vector field $\tilde{X}$ on $E$ such that $\tilde{X}_e\in\mathcal{H}_e$ and $\mathrm{d}_e\pi \big(\tilde{X}_e\big)=X\big(\pi(e)\big),$
for all $e\in E$.
\end{definition}

The horizontal lift of vector fields is therefore $\mathcal C^\infty(M)$-linear operator,
$\mathfrak X(M) \to \mathfrak X(E)$. The horizontal space $\mathcal H$ at any point $p\in M$ is spanned by the values of horizontal lifts of vector fields in $M$. Therefore, the horizontal lift operator completely determines the horizontal distribution $\mathcal H$. This operator is also characterized by the following properties:
\begin{itemize}
    \item[(a)] It is $\mathcal C^\infty(M)$-linear.
    \item[(b)] For any $\pi$-projectable vector field $Y$ in $E$, $\widetilde{\pi_*(Y)} = Y$.
\end{itemize}
Up to now, we have seen four different ways of determining a connection in a submersion $\pi\colon E\to M$.
\begin{enumerate}
    \item A smooth distribution $\mathcal H$ of horizontal vector spaces in $E$;
    \item A covariant derivative operator
    $$\nabla \colon \Gamma(E,M)\to \bigcup_{s\in \Gamma(E,M)} \Omega^1(M, s^*({\rm Ver}(E)));$$
    \item A section $\sigma$ of the $1$-jet bundle $\pi_{1,0}\colon J^1E\to E$;
    \item An horizontal lift operator;
    $${\rm lift}^{\mathcal H}\colon \mathfrak X(M)\to \mathfrak X(E),\quad X\mapsto \tilde X.$$
\end{enumerate}
As each one of this objects determines the other three, we will use the one that best fits our purposes depending on the situation. \\

Let $\mathcal H$ be an connection on $\pi\colon E\to M$ and $\mathfrak X(M)\to \mathfrak X(E)$, $X\mapsto \Tilde X$ its corresponding horizontal lift operator. For a pair of vector fields $X,Y\in \mathfrak X(M)$, the formula
\begin{equation}\label{def:curvatureForm}
    R(X,Y) = \frac{1}{2}([\tilde X,\tilde Y] - \widetilde{[X,Y]}) 
\end{equation}
defines a \emph{vertical} vector field, that is, it projects by $\pi$ onto zero. The restriction of $R(X,Y)$ to a fiber $E_x$ depends only on the values $X_x,Y_x\in T_xM$. If follows that $R$ is a $2$-form in $M$ with values in the 
infinite dimensional\footnote{This infinite dimensional vector bundle $\mathfrak X_{{\rm Ver}(E)}\to M$ can be seen as the inductive limit of all vector bundles $L\to M$ where for all $x\in M$ the fiber $L_x$ is a finite dimensional vector space of $\mathfrak X(E_x)$.} vector bundle $\mathfrak X_{\rm Ver(E)}\to M$ whose fiber over $x\in M$ is $\mathfrak X(E_x)$.

\begin{definition}\label{df:curvature}
We call \emph{curvature form} of $\mathcal H$ to the $2$-form $R\in \Omega^2(M,\mathfrak X_{{\rm Ver}(E)})$ defined in \ref{def:curvatureForm}.
The connection $\mathcal H$ is said to be \emph{flat} if its curvature form vanishes. 
\end{definition}

Note that the curvature $R$ is the obstruction for the horizontal lift operator to be a Lie algebra morphism, and to the existence of horizontal local sections. If a connection is flat then for any $p\in E$ there exists a unique germ $s$ of horizontal section such that $s(\pi(p)) = p$ defined in a neighborhood of $\pi(p)$ . 

\subsection{Holonomy}

\begin{definition}
Let $\pi:E\to M$ be a smooth submersion endowed with a connection $\mathcal{H}$. Given a smooth
curve $\gamma:I\to M$, a
{\em horizontal lift\/} of $\gamma$ is a smooth
curve $\tilde\gamma:I\to E$ with $\pi \circ\tilde\gamma=\gamma$ and such that $\tilde\gamma'(t)\in\mathcal{H}_{\tilde\gamma(t)}$,
for all $t\in I$.  
\end{definition}

\begin{remark}\label{rm:vector_holonomy}
The lifting of smooth curves is compatible with the lifting of vector fields. That is, if $\gamma$ is an integral curve for the vector field $X\in\mathfrak X(M)$, then any horizontal lift of $\gamma$ is an integral curve of $\tilde X$, and reciprocally, any integral curve $\tilde\gamma$ of $\tilde X$ such that $\pi(\tilde\gamma(0)) = \gamma(0)$ is a horizontal lift of $\gamma$.
\end{remark}

\begin{definition}
We say that $\mathcal H$ is an \emph{connection} if it satisfies any of the following equivalent properties:
\begin{enumerate}
\item[(a)] For any complete vector field $X$ in $M$ its horizontal lift $\tilde X$ is a complete vector field in $E$.
\item[(b)] For any smooth (and therefore any piece-wise smooth) curve $\gamma$ in $M$ and any point $e\in E_{\gamma(0)}$ there exists a horizontal lift $\tilde\gamma$ of $\gamma$ with $\tilde\gamma(0) = p$.
\end{enumerate}
\end{definition}

The equivalence of properties (a) and (b) is clear. If $X$ is a complete vector field in $M$, for any $p\in E$ the exponential $\exp(t\tilde X)(p)$ is defined as the horizontal lift of $\exp(tX)(\pi(p))$ with initial condition $p$. Reciprocally, any smooth curve $\gamma$ is a concatenation of segments of integral curves of complete vector fields. By integrating the lifts of such vector fields we obtain the lift of $\gamma$.

\begin{remark}
We warn the reader that in several references the term Ehresmann conection is used as a synonym of connection. However, as C. Ehresmann required the existence of arbitrary horizontal lifts his definition of infinitesimal connection \cite[p. 157]{ehresmann1952} we prefer to apply the term Ehresmann connection only to connections with arbitrary horizontal lifts. Some other authors use the terms \emph{complete} (see \cite[page 81]{Kolar}) or \emph{geometric Painlevé property} (see, for instance, \cite{iwasaki2006}).     
\end{remark}

\begin{example}
Any connection in a proper smooth submersion is an Ehresmann connection.
Linear connections, principal connections and associated connections are Ehresmann connections, see \cite[Ch. II]{kobayashi}.
\end{example} 

Let $\pi:E\to M$ be a smooth submersion endowed with an Ehresmann connection $\mathcal{H}$. By a \emph{path} in $M$ we mean a continuous piece-wise smooth map from $I = [0,1]$ to $M$. Let us denote by $\ell(M)$ the space of paths in $M$. For each $x$ and $y$ in $M$ let us denote by $\ell(M,x,y)$ the space of paths $\gamma$ in $M$ such that $\gamma(0)=x$ and $\gamma(1) =y$. There is a concatenation law:
$$\ell(M,y,z) \times \ell(M,x,y)\to \ell(M,x,z), \quad (\gamma_1,\gamma_2)\mapsto \gamma_1\star\gamma_2,$$
where
$$(\gamma_1\star\gamma_2)(t) = \begin{cases}
\gamma_2(2t), & 0\leq t \leq \frac{1}{2};\\ \gamma_1(2t-1), & \frac{1}{2} < t \leq 1.
\end{cases}$$
A path $\gamma \in  \ell(M,x,y)$ defines a diffeomorphism $$\mathrm{hol}_{\gamma}^{\mathcal{H}}:E_x \to E_y,\quad p \mapsto \tilde{\gamma}(1),$$ 
where $\tilde{\gamma}$ denotes the horizontal lift of $\gamma$ with $\tilde\gamma(0) = e\in E_x$. These holonomy maps are compatible with the concatenation of paths,
$${\rm hol}_{\gamma_1\star\gamma_2}^{\mathcal{H}} = {\rm hol}_{\gamma_1}^{\mathcal{H}}\circ {\rm hol}_{\gamma_2}^{\mathcal{H}}.$$
In particular, if we set the two end points to coincide we have a representation:
$${\rm hol}^{\mathcal H}\colon \ell(M,x,x)\to {\rm Diff}(E_x).$$
There are three relevant subgroups of ${\rm Diff}(E_x)$ related with the holonomy representation.
\begin{enumerate}
    \item[(a)] The image of ${\rm hol}^{\mathcal H}$ is the so-called \emph{holonomy group} ${\rm Hol}_x^{\mathcal H}$ of $\mathcal H$ at the point $x$.
    \item[(b)] The image by ${\rm hol}^{\mathcal H}$ of the set of homotopically trivial loops is the so-called \emph{restricted holonomy group} ${\rm RHol}_x^{\mathcal H}$ of $\mathcal H$ at the point $x$.
    \item[(c)] Note that if $\mathcal U$ is a neighborhood of $x$ then we have an inclusion $\ell(\mathcal U,x,x)\subset \ell (M,x,x)$. The intersection, 
    $${\rm LHol}^{\mathcal H}_x = \bigcap_{\mathcal U\in \mathcal{N}} {\rm hol}^{\mathcal H}(\ell(\mathcal U,x,x))$$
    where $\mathcal{N}$ is any basis of simply connected neighborhoods of $x$, receives the name of \emph{local holonomy group}.
\end{enumerate}

Clearly, we have a chain of inclussions:
$${\rm LHol}^{\mathcal H}_x \subseteq {\rm RHol}^{\mathcal H}_x  \subseteq {\rm Hol}^{\mathcal H}_x.$$

\begin{theorem}[Weak form of Ambrose-Singer]\label{th:w_AS}
Let $\pi\colon E\to M$ be a smooth submersion endowed with an Ehresmann connection $\mathcal H$. Let $x$ be a point of $M$ and $X = R(\vec v_1,\vec v_2)\in \mathfrak X(E_x)$ a vector field in the image of the curvature tensor. Then $X$ is a complete vector field in $E_x$ and for any $t\in \mathbb R$ we have  $\exp(tX)\in {\rm LHol}^{\mathcal H}_x$.
\end{theorem}

\begin{proof}
Let us fix a neighborhood $\mathcal U$ of $x$. We may assume that $\vec v_1$ and $\vec v_2$ do not vanish, otherwise we obtain $X = 0$.
Let us consider $Y_1$ and $Y_2$ a pair of commuting vector fields in $M$ such that $Y_1(x) = \vec v_1$ and $Y_2(x) = \vec v_2$. Therefore $[\tilde Y_1,\tilde Y_2]$ is a vertical vector field and $2X = [\tilde Y_1,\tilde Y_2]|_{E_x}$.
Let $\varepsilon>0$ small enough (so as to allow application inside $\mathcal U$ of the flows of $Y_1$ and $Y_2$ to points near $x$ appearing in the following formula). By an immediate corollary of \cite[Theorem 3.16, p. 21]{Kolar} (cf. \cite[Exercise 6, p. 78]{Warner}) we have for any $p\in E_x$:
$$[\tilde Y_1, \tilde Y_2](p) = \left.\frac{\mathrm{d}}{\mathrm{d}\varepsilon}\right|_{\varepsilon = 0} ({\rm exp}(-\sqrt{\varepsilon} \tilde Y_2) \circ {\rm exp}(-\sqrt{\varepsilon}\tilde Y_1) \circ {\rm exp}(\sqrt{\varepsilon} \tilde Y_2) \circ {\rm exp}(\sqrt{\varepsilon} \tilde Y_1))(p)$$
Let $\gamma_{\varepsilon}$ be the piece-wise smooth loop based in $x\in M$ defined 
by: 
$$\gamma_\varepsilon(t) = \begin{cases} \exp\left(4t \sqrt{\varepsilon}  Y_1\right)x,\;  &t \in [0,\tfrac{1}{4}],\\
\exp \left(\sqrt{\varepsilon}(4t-1)Y_2\right)x_1, \; &t \in [\tfrac{1}{4},\tfrac{1}{2}],\\
\exp\left(-4\sqrt{\varepsilon}(t-\tfrac{1}{2}) Y_1\right)x_2, \;  &t \in [\tfrac{1}{2},\tfrac{3}{4}],\\
\exp\left(-4\sqrt{\varepsilon}(t-\tfrac{3}{4}) Y_2\right)x_3, \; &t \in [\tfrac{3}{4},1],\\
\end{cases}$$
where $x_1 = \exp(\sqrt{\varepsilon}Y_1)x$, $x_2 = \exp(\sqrt{\varepsilon}Y_2)x_1$ and $x_3 = \exp(-\sqrt{\varepsilon}Y_1)x_2$. 
By definition of horizontal lifting we have that the horizontal lifting of $\gamma_{\varepsilon }$ with initial condition $\tilde \gamma_{\varepsilon }(0) = p$ is:
$$\tilde \gamma_\varepsilon(t) = \begin{cases} \exp\left(4t \sqrt{\varepsilon}  \tilde Y_1\right)p,\;  &t \in [0,\tfrac{1}{4}],\\
\exp \left(\sqrt{\varepsilon}(4t-1)\tilde Y_2\right)p_1, \; &t \in [\tfrac{1}{4},\tfrac{1}{2}],\\
\exp\left(-4\sqrt{\varepsilon}(t-\tfrac{1}{2}) \tilde Y_1\right)p_2, \;  &t \in [\tfrac{1}{2},\tfrac{3}{4}],\\
\exp\left(-4\sqrt{\varepsilon}(t-\tfrac{3}{4}) \tilde Y_2\right)p_3, \; &t \in [\tfrac{3}{4},1],\\
\end{cases}
$$
where $p_1 = \exp(\sqrt{\varepsilon}Y_1)p$, $p_2 = \exp(\sqrt{\varepsilon}Y_2)p_1$ and $p_3 = \exp(-\sqrt{\varepsilon}Y_2)p_1$. From the compatibility of the horizontal lifting of vector fields and its integral curves we have:
$${\rm hol}^{\mathcal{H}}(\gamma_{\varepsilon})(p) = ({\rm exp}(-\sqrt{\varepsilon} \tilde Y_2) \circ {\rm exp}(-\sqrt{\varepsilon}\tilde Y_1) \circ {\rm exp}(\sqrt{\varepsilon} \tilde Y_2) \circ {\rm exp}(\sqrt{\varepsilon} \tilde Y_1))(p)$$
for all $p\in E_x$ and therefore
$${\rm hol}^{\mathcal{H}}(\gamma_{\varepsilon}) = {\rm exp}(-\sqrt{\varepsilon} \tilde Y_2) \circ {\rm exp}(-\sqrt{\varepsilon}\tilde Y_1) \circ {\rm exp}(\sqrt{\varepsilon} \tilde Y_2) \circ {\rm exp}(\sqrt{\varepsilon} \tilde Y_1)$$
Putting those formulas together we obtain that for $\varepsilon>0$:
$$[\tilde Y_1, \tilde Y_2](p) = \left.\frac{\mathrm{d}}{\mathrm{d}\varepsilon}\right|_{\varepsilon = 0} {\rm hol}^{\mathcal{H}}(\gamma_{\varepsilon})(p)$$
and thus for any $2\varepsilon>t>0$ we have,
$${\rm exp}(tX) = {\rm hol}^{\mathcal{H}}(\gamma_{t/2}).$$
this proves that $X$ is complete and its flow is in the image of $\ell(\mathcal U,x,x)$ by the holonomy representation. This holds for any neighborhood $\mathcal U$ of $x$ and therefore $\exp(tX)\in {\rm LHol}_x^{\mathcal H}$ (for small positive $t$ and thus for all $t\in\mathbb R$).
\end{proof}

\section{Group bundles}\label{s_GB}

For practical reasons it is convenient to introduce Lie groupoids in general and adress group bundles as a particular case. We consider the standard definition of Lie groupoid as in \cite[p. 112]{moerdijk2003introduction} (Note that this definition corresponds to the term differentiable groupoid for some authors cf. \cite{mackenzie1987lie}). This is, 
a Lie groupoid $\mathcal G\rightrightarrows M$ consists of the following data,
\begin{enumerate}
    \item A pair of smooth submersions $s,t\colon \mathcal G\to M$ called source and target respectively.
    \item Smooth maps 
    $${\mathfrak m}\colon \mathcal G\,{}_s\hspace{-1mm}\times_t \mathcal G\to \mathcal G : (\sigma,\tau)\mapsto \sigma \tau, \quad {\mathfrak i}\colon \mathcal G\to \mathcal G: \sigma \mapsto \sigma^{-1} \quad {\mathfrak e}\colon M\to \mathcal G: x\mapsto \mathfrak{e}_x$$ called multiplication, inversion and identity, satisfying the standard groupoid axioms.
\end{enumerate}

For $x, y\in M$ we set the following notation:
$$\mathcal G_{xy} = s^{-1}(x)\cap t^{-1}(y), \quad \mathcal G_{x\bullet} = s^{-1}(x), \quad
\mathcal G_{\bullet y} = t^{-1}(y).$$
For each $x\in M$ the fiber $\mathcal G_{xx}$ is a Lie group, called the stabilizer group of $x$. If $\mathcal G_{xy}\neq \emptyset$ then $\mathcal G_{xx}$ and $\mathcal G_{yy}$ are isomorphic Lie groups. The manifold $\mathcal G_{xy}$ is a left principal homogeneous $\mathcal G_{yy}$-space and a right principal homogeneous $\mathcal G_{xx}$-space. Left and right translations are partially defined. If $\sigma\in \mathcal G_{xy}$ then we have:
$$L_\sigma \colon \mathcal G_{\bullet x}\to \mathcal G_{\bullet y}, \quad R_\sigma \colon \mathcal G_{y\bullet}\to \mathcal G_{x\bullet}.$$

\begin{definition}
A Lie groupoid is called transitive if the map $(s,t)\colon \mathcal G\to M^2$ is surjective. 
\end{definition}

Let us define ${\rm Lie}(\mathcal G)$ as the pullback through the identity section $\mathfrak e$ of the vertical space to the target map; 
${\mathfrak e}^*(\ker({\rm d} t))$. It is as a vector bundle over $M$ 
and its fiber on a point $x\in M$ is
$T_{\mathfrak e_x}\mathcal G_{\bullet x}$. The differential of the source map is then a vector bundle morphism,
$$\rho\colon {\rm Lie}(\mathcal G) \to TM, \quad \vec v\mapsto {\rm d} s(\vec v)$$
that we call the \emph{anchor} map. Sections $A$ of ${\rm Lie}(\mathcal G)$ can be extended to left invariant vector fields in $\mathcal G$ by setting $A(g) = {\rm d} L_g(A(s(g))$. Lie bracket of left invariant vector fields are left invariant vector fields. This realizes $\Gamma({\rm Lie}(\mathcal G))$ as a Lie subalgebra of $\mathfrak X(\mathcal G)$ and defines a bracket operator in sections, for any open subset $U$ of $M$,
$$\Gamma({\rm Lie}(\mathcal G); U)\times \Gamma({\rm Lie}(\mathcal G); U) \to \Gamma({\rm Lie}(\mathcal G); U).$$
This bracket is related with the anchor map,
$$[A,fB] = (\rho(A)f)B + f[A,B], \quad f\in \mathcal C^\infty(U), \quad A,B\in \Gamma({\rm Lie}(\mathcal G); U).$$
This structure (vector bundle, bracket in sections, anchor) is known as a \emph{Lie algebroid} and therefore we call ${\rm Lie}(\mathcal G)$ the Lie algebroid of $\mathcal G$.

\subsection{Group bundles}

\begin{definition}
A Lie groupoid $\mathfrak{G}\rightrightarrows M$ is called a group bundle if $s=t$. 
\end{definition}

Given a group bundle, there is no difference between source and target, so that we may use the same symbol $\pi\colon \mathfrak{G}\to M$ for either of them and the notation $\mathfrak{G}_x = \pi^{-1}(x)$ for the stabilizer group of $x\in M$. Alternatively, we may define {\em{group bundle\/}} as a smooth bundle $\pi:\mathfrak{G} \to M$ endowed with a smooth map $\mathfrak{m}:\mathfrak{G}\times_M \mathfrak{G}\to \mathfrak{G}$ such that for each $x\in M$,
the map $\mathfrak{m}_x: \mathfrak{G}_x \times \mathfrak{G}_x \to \mathfrak{G}_x$ turns $\mathfrak{G}_x$ into a Lie Group. A group bundle can be thought as a {\em{smoothly varying\/}} family $\{\mathfrak{G}_x\}_{x\in M}$ of Lie groups parameterized by the points of $M$. 

If $\mathfrak{G}\to M$ is a group bundle, then the Lie algebroid ${\rm Lie}(\mathfrak{G})\to M$ has trivial anchor map. The value at a point $x\in M$ of a bracket depends only of the values at $x$ of the involved section. Therefore ${\rm Lie}(\mathfrak{G})\to M$ is a vector bundle by Lie algebras and ${\rm Lie}(\mathfrak G)_x = {\rm Lie}(\mathfrak G_x)$. \\

\begin{example}
If $\mathcal G\rightrightarrows M$ is a Lie groupoid then the equalizer of the source and target maps ${\rm eq}(s,t)\to M$ is a group bundle (and a Lie subgroupoid of $\mathcal G$). Its fiber over $x\in M$ is the stabilizer group $\mathcal G_{xx}$.
\end{example}

\begin{example}
Each vector bundle of rank $n$ is a group bundle. In fact note that in this case each fiber $E_x$ is isomorphic to the additive Lie group $(\mathbb{R}^n,+)$.
\end{example}

Let $\pi:\mathfrak{G} \to M$ be a group bundle. As usual 
we will denote by $\Gamma(\mathfrak{G};U)$ the set of all smooth local sections of $\pi$ defined on the open set $U$.
Given local sections $\epsilon_1, \epsilon_2 \in \Gamma(\mathfrak{G};U)$, their product is defined by 
\[
(\epsilon_1 \cdot \epsilon_2)(x) =\mathfrak{m}(\epsilon_1(x),\epsilon_2(x)).
\]Note that $\epsilon_1 \cdot \epsilon_2 \in \Gamma(\mathfrak{G};U)$. Therefore, this product turns the set $\Gamma(\mathfrak{G};U)$ into a group, with identity element $\mathfrak e|_U$. \\

\begin{definition}
Let $\pi:\mathfrak{G}\to M$, $\pi': \mathfrak{G}'\to M$ be group bundles.
A smooth map $\phi: \mathfrak{G}\to \mathfrak{G}'$ is said to be a
{\em group bundle morphism\/} if it is {\em fiber preserving,\/} and  for all $x\in M$ the map $\phi|_x: \mathfrak{G}_x\to \mathfrak{G}'_x,$ is a group homomorphism. 
\end{definition}

A group bundle isomorphism is, by definition, a group bundle morphism that is also a diffeomorphism. This ensures that the maps between the fibers are Lie group isomorphisms.


\subsection{Locally trivial group bundles}

A group bundle $\pi:\mathfrak{G}\to M$ is said to be {\em{locally trivial\/}}, if for each $x\in M$ there is a neighborhood $x\in U\subset M$ and a fiber preserving diffeomorphism $\phi:\mathfrak{G}|_U \to U \times \mathfrak{G}_x$ such that for each $y\in U$ the map $\phi_y : \mathfrak{G}_y \to \{y\} \times \mathfrak{G}_x$ is a Lie group isomorphism. \\

Let us recall that, given a Lie group $G$, its identity component $G^0$ is a clopen Lie subgroup. Let us see that the identity component of a group bundle is the union of the identity components of the fibers.

\begin{lemma}\label{lm:cc}
Let $\pi:\mathfrak{G} \to M$ be a locally trivial group bundle with connected base $M$ and $
\mathfrak{G}^0$ be the connected component of $\mathfrak{G}$ containing the image of the identity section $\mathfrak e$. Then, 
$$\mathfrak{G}^0 = \bigcup_{x\in X} (\mathfrak{G}_x)^0.$$
\end{lemma}

\begin{proof}
Let us consider $\mathfrak{G}' =\bigcup_{x\in X} (\mathfrak{G}_x)^0$. We will see that $\mathfrak{G}'$ is a connected closed and open subset of $\mathfrak{G}$ containing $\mathfrak{G}^0$, then by the minimality of $\mathfrak{G}^0$ we conclude that $\mathfrak{G}' = \mathfrak{G}^0$.

In order to obtain the result, it is enough to see that for an open covering $\{\mathcal U_{i}\}_{i\in I}$ of $M$ we have that $\mathfrak{G}'|_{\mathcal U_i}$ is clopen in $\mathfrak{G}|_{\mathcal U_i}$ and $\mathfrak{G}'|_{\mathcal U_i} \subset \mathfrak{G}^0|_{\mathcal U_i}$.

We may take this open covering $\{\mathcal U_i\}_{i\in I}$ in such a way that each $\mathcal U_i$ is a trivializing connected neighborhood. We also consider smooth local trivializations $\psi_i\colon \mathfrak{G}|_{\mathcal U_i}\xrightarrow{\sim} \mathfrak{G}_i\times \mathcal U_i$. Then, there is a connected component $\mathfrak{G}_i^0$ of $\mathfrak{G}_i$ such that $G_i^0\times \mathcal U_i$ contains the image of $\psi_i\circ \mathfrak e|_{\mathcal U_i}$. It follows that $\mathfrak{G}'|_{\mathcal U} = \psi_i^{-1}(\mathfrak{G}_i^0\times \mathcal U)$ is a clopen subset of $\mathfrak{G}|_{\mathcal U_i}$ containing the image of $\mathfrak e|_{\mathcal U_i}$ and therefore $\mathfrak{G}'|_{\mathcal U}\subset \mathfrak{G}^0|_{\mathcal U}$. Thus, $\mathfrak{G}'$ is a clopen subset of $G^0$ and they are equal.
\end{proof}

\begin{remark}
A group bundle may be locally trivial as fiber bundle but not as a group bundle. A local trivialization in the general sense of fiber bundles (as in \cite{kobayashi,Kolar}) is not necessarily a local trivialization in the sense of group bundles. 
\end{remark}

\begin{example}\label{ejemplo:fibradonotrivial}
For each pair of real numbers $(\lambda,\mu) \in \r^2$ we have a semidirect product $\r^{+} \,\rtimes _{\lambda,\mu} \; \r^{2}$ given by the formula:
\[
(a,b,c)\cdot (a', b', c') := (aa', b+ a^{\lambda}b', c+a^{\mu}c').
\]
This family of Lie groups can be seen as a group bundle $\r^{+} \times \r^{4} \to \r^{2}$ 
with fiberwise multiplication, 
\[
\mathfrak{m}\left(\big(a,b,c,\lambda,\mu \big), \big(a',b', c', \lambda,\mu\big)\right) = (aa', b + a^{\lambda}b', c +a^{\mu}c', \lambda, \mu).
\]
This group bundle is not locally trivial because all its fibers are not isomorphic. 
\end{example}

\subsection{Group connections}

Let $\pi \colon \mathfrak{G}\to M$ be a group bundle. Its jet bundle $\pi_{1,-1}\colon J^1\mathfrak{G}\to M$ is naturally endowed with a group bundle structure where 
$$(j_x^1g)\cdot (j_x^1h) = j_x^1(g\cdot h).$$ The projection,
$\pi_{1,0}: J^1\mathfrak{G}\to \mathfrak{G}$ is a surjective group bundle morphism and its kernel, that we denote $K^1\mathfrak{G}$, consist of the jets of sections $j^1_xg$ such that $g(x)={\mathfrak e}_x$. We have an exact sequence of group bundles,
\begin{equation}\label{eq:se_jet}
0\to K^1\mathfrak{G} \to J^1\mathfrak{G} \xrightarrow{\pi_{1,0}} \mathfrak{G} \to 0.
\end{equation}

\begin{definition}\label{df:groupc}
A group connection in $\pi\colon \mathfrak{G}\to M$ is a connection $\mathcal H$ satisfying the following equivalent conditions:
\begin{enumerate}
    \item[(a)] The induced jet bundle section $\sigma\colon \mathfrak{G}\to J^1\mathfrak{G}$ is a group bundle morphism.
    \item[(b)] For all $(g,h)\in \mathfrak{G}\times_M \mathfrak{G}$, 
        $$\mathrm{d}_{(g,h)}\mathfrak m(\mathcal H_g \times_{T_xM} \mathcal H_h) = \mathcal H_{g\cdot h}, \quad 
        \mathrm{d}_g\mathfrak i(\mathcal H_g) = \mathcal H_{g^{-1}}.$$
\end{enumerate}
\end{definition}

In what follows, let $\pi\colon \mathfrak{G}\to M$ be a group bundle with connected base $M$ and $\mathcal H$ a group connection in $\mathfrak{G}$.

\begin{remark}
The following elementary facts concerning to horizontal lifts of smooth curves are direct consequences of condition (b) in Definition \ref{df:groupc}. Let $\gamma\colon I\to M$ be a smooth curve. 
\begin{enumerate}
    \item Let us assume that $\gamma$ is an integral curve of a vector field $X$ in $M$. Let $\tilde X$ be the $\mathcal H$-horizontal lift of $X$. Let $\hat \gamma\colon I\to \mathfrak{G}$ and $\tilde \gamma\colon I\to \mathfrak{G}$ be two integral curves of $\tilde X$ such that $\pi\circ \tilde\gamma = \pi\circ \hat\gamma = \gamma$. Then,
    $$\tilde \gamma\cdot \hat\gamma = \mathfrak m\circ (\tilde\gamma,\hat\gamma)\quad \mbox{and}\quad \tilde\gamma^{-1} = \mathfrak i\circ \tilde\gamma$$
    are integral curves of $\tilde X$ and $\mathcal H$-horizontal lifts of $\gamma$.
    \item The pullback $\gamma^*\mathfrak{G} = \mathfrak{G}\times_M I \to I$ is a group bundle endowed with a group connection $\gamma^*\mathcal H$.
    \item A lift $\tilde\gamma$ of a $\gamma$ is $\mathcal H$-horizontal if and only if the curve $t\in I\mapsto (\tilde\gamma(t), t)\in \gamma^*\mathfrak{G}$ is an integral curve of the $\gamma^*\mathcal H$-horizontal lift to $\gamma^*\mathfrak{G}$ of the vector field $\frac{\partial}{\partial t}$ in $I$.
    \item If $\tilde\gamma$ and $\hat\gamma$ are $\mathcal H$-horizontal lifts of $\gamma$, the curves $\tilde \gamma^{-1}\colon t\mapsto \tilde\gamma(t)^{-1}$ and $\tilde\gamma\cdot\hat\gamma \colon t\mapsto \tilde\gamma(t)\cdot \hat\gamma(t)$ are $\mathcal H$-horizontal lifts of $\gamma$.
\end{enumerate}
\end{remark}

\medskip 
\begin{lemma}[of analytic continuation]\label{lm:ancont}
Let $\gamma\colon \mathbb R\to M$ be smooth curve and $x=\gamma(0)$. There are $\varepsilon >0$ and a group bundle trivialization of $\gamma^*\mathfrak{G}|_{(-\varepsilon,\varepsilon)}\to (-\varepsilon,\varepsilon)$,
\begin{equation}\label{eq:bundletrivialization}
    \psi\colon \mathfrak{G}_{x}\times (-\varepsilon, \varepsilon) \xrightarrow{\sim} 
\gamma^*\mathfrak{G}|_{(-\varepsilon,\varepsilon)},
\end{equation}
such that $\psi_*(\mathcal H_0) = \gamma^*\mathcal H$ where $\mathcal H_0 = \{0\}\times T(-\varepsilon,\varepsilon)$ stands for the standard horizontal connection in the cartesian product $\mathfrak{G}_x\times (-\varepsilon,\varepsilon)$.
\end{lemma}

\begin{proof}
Let $\vec X$ be the $\mathcal H$-horizontal lift of $\frac{\partial}{\partial t}$ to $\gamma^*\mathfrak G$. Note that condition (a) in Definition \ref{df:groupc} implies that if $g(t)$ and $h(t)$ are integral curves of $\vec X$ then $g(t)\cdot h(t)$ and $h(t)^{-1}$ are also integral curves of $\vec X$.

First, let us assume that the fibers of $\pi$ are connected. By means of the local reduction theorem there are neighborhoods $\mathcal U_0$ of $(\mathfrak e_x,0)$ in $\gamma^*\mathfrak{G}$ and $\mathcal U'_0$ of $\mathfrak e_x$ in $\mathfrak{G}_x$ and a difeomorphism,
$$\psi_0\colon   \mathcal U'_0 \times (-\varepsilon,\varepsilon) \xrightarrow{\sim} \mathcal U_0$$
such that $(\psi_0)_*\left(\frac{\partial}{\partial t}\right) = \vec X$.
Let us define inductively 
$$\mathcal U_{n+1} = {\mathfrak m}(\mathcal U_n, \mathcal U_n), \quad \mathcal U_{n+1}' = {\mathfrak m}(\mathcal U'_n, \mathcal U'_n),$$
we may extend also inductively the domain of $\psi_0$ by defining inductively,
$$\psi_{n+1}(g\cdot h,t) = \mathfrak{m}(\psi_n(g,t),\psi_n(h,t)).$$
since any connected Lie group is spanned by any open neighborhood of the identity element, in the limit $n\to \infty$ we obtain a group bundle trivialization as in \ref{eq:bundletrivialization}.

In the general case, by the above argument we have a trivialization of the connected component $\gamma^*\mathfrak G^0|_{(-\varepsilon,\varepsilon)}\to (-\varepsilon,\varepsilon)$. Let us consider $g_x\in \mathfrak{G}_x$ and let $g(t)$ be the integral curve of $\vec X$ with initial condition $g(0) = g_x$. Let us see that $g(t)$ is well defined for all $t\in (-\varepsilon,\varepsilon)$. Let $\delta$ be the supreme in $(-\varepsilon,\varepsilon)$ such that $g(t)$ is defined in $(-\delta,\delta)$. Reasoning by the contradiction, let us assume $\delta<\varepsilon$. Then, let us consider $y = \gamma(\delta)$ and $h_y$ any element of $\mathfrak{G}_y$ in the same connected component of $\mathfrak{G}$ that $g_x$. There is an integral curve $h(t)$ of $\vec X$ with initial condition $h(\delta) = h_y$ defined for $t$ varying in an open neighborhood of $\delta$. Let $\delta'\in \mathbb R$ in the common domain of definition of $g(t)$ and $h(t)$.
Then, $h(\delta')^{-1} g(\delta')$ is in the connected component $\mathfrak G^0$. Let us consider the integral curve $\sigma(t)$ of $\vec X$ with initial condition $\sigma(\delta') = h(\delta')^{-1} g(\delta')$. By the trivialization of the connected component we know that the domain of definition of $\sigma$ is at least $(-\varepsilon,\varepsilon)$. Then we obtain an extension of $g(t)$,
$$\tilde g(t) = \begin{cases}
g(t) & t\in (-\delta,\delta) \\ 
h(t)\sigma(t) & t> \delta'
\end{cases}$$
by the same argument we also obtain an extension of $g(t)$ on the left side of its domain, in contradiction with the maximality of $\delta<\varepsilon$. 
Thus, the domain of definition of $g(t)$ is $(-\varepsilon,\varepsilon)$. The flow ${\rm exp}(t\vec X)\cdot g_x$ is well defined for any $g_x\in \mathfrak{G}_x$ and $t\in (-\varepsilon,\varepsilon)$, We obtain the trivialization as in \eqref{eq:bundletrivialization} by setting
$\psi(g_x,t) = {\rm exp}(tX)\cdot g_x$.
\end{proof}

\begin{proposition}
Let $\mathcal H$ be a group connection in $\pi\colon \mathfrak{G}\to M$. Then
\begin{itemize}
    \item[(a)] $\mathcal H$ is an Ehresmann connection. 
    \item[(b)] For any path $\gamma$ joining two points $x$ and $y$ of $M$
$${\rm hol}_\gamma^{\mathcal H}\colon \mathfrak{G}_x \to \mathfrak{G}_y$$
is a Lie group isomorphism. 
\end{itemize}
\end{proposition}

\begin{proof}
Let $\gamma$ be a smooth curve defined in some neighborhood of the closed interval $[0,1]$. $\gamma\colon (-\delta,1+\delta) \to M$. By application of Lemma \ref{lm:ancont} for each $s\in [0,1]$ there exist a $\varepsilon(s)$ and a group bundle trivialization:
$$  \psi_s\colon \mathfrak{G}_{\gamma(s)}\times (s-\varepsilon(s), s+\varepsilon(s)) \xrightarrow{\sim} 
\gamma^*\mathfrak{G}|_{(s-\varepsilon(s),s+\varepsilon(s))}.$$
As the interval $[0,1]$ is compact there is a finite number 
$s_0 = 0< s_1< \ldots <s_{n-1}<s_n = 1$ such that the family 
$\{I_j\}_{j=0}^n$ with $I_j = (s_j-\varepsilon(s_j),s_j+\varepsilon(s_j))$ covers the interval $[0,1]$.  If this family is minimal we also have that
$I_{j-1} \cap I_j \neq \emptyset$
for any $j=1,\ldots,n$ we take $t_j$ for $j=1,\ldots n$ 
in such intersection and define $t_0 = 0$ and $t_{n+1}=1$ so we 
have a sequence of numbers:
$$s_0 = t_0 = 0 < t_1 < s_1 < t_2 <\ldots < s_{n-1} < t_n < s_n = t_{n+1} = 1.$$
Let us consider the projections:
$$\phi_j = \pi_1\circ \psi_{s_j}^{-1} \colon \gamma^*\mathfrak{G}|_{(s_j-\varepsilon(s_j),s_j+\varepsilon(s_j))} \to \mathfrak{G}_{\gamma(s_j)}.$$
For an initial condition $g_x\in \mathfrak{G}_x$ we set $g_0 = g_x$, $g_1 = \psi_0(g_0,t_1)$, and define inductively for $k = 1,\ldots, n$:
$$g_{j+1} = \psi_{j}(\phi_j(g_j),t_{j+1}).$$
Then we have that the curve:
$$\tilde\gamma(t) = \begin{cases}
\psi_{0}(g_0,t) & t\in I_0 \\
\psi_{s_j}(\phi(g_j),t) & t\in I_j\,\, \mbox{with}\,\, j = 1,\ldots,n
\end{cases}$$
is an horizontal lift of $\gamma$ with initial condition $\tilde\gamma(0) = g_x$. If follows that $\mathcal H$ is an Ehresmann connection. 

Let us see (b). Any path decomposes as a concatenation of smooth curves so that it suffices to prove the statement for smooth curves. As before, let $\gamma$ be a smooth curve defined in $(-\delta,1+\delta)$. Let $g$ and $h$ be two elements of $\mathfrak{G}_{\gamma(0)}$. Let us consider $\tilde\gamma$ and $\hat\gamma$ the $\mathcal H$-horizontal lifts of $\gamma$ with initial conditions $\tilde\gamma(0)=g$ and $\hat\gamma(0) = h$. Then ${\rm hol}_\gamma^\mathcal H(g) = \tilde\gamma(1)$ and ${\rm hol}_\gamma^\mathcal H(h) = \hat\gamma(1)$. The curves $t\mapsto (\tilde\gamma(t),t)$ and $t\mapsto (\hat\gamma(t),t)$ are integral curves of the 
$\gamma^*\mathcal H$-horizontal lift $\vec X$ of $\frac{\partial}{\partial t}$. Therefore $t\mapsto (\tilde\gamma(t)\cdot \hat\gamma(t),t)$ is also an integral curve of $\vec X$ and $t\mapsto \tilde \gamma(t)\cdot \hat\gamma(t)$ is an $\mathcal H$-horizontal lift of $\gamma$. If follows,
$${\rm hol}^{\mathcal H}_\gamma(g\cdot h) = {\rm hol}^{\mathcal H}_\gamma(g)\cdot {\rm hol}^{\mathcal H}_\gamma(h).$$
Thus ${\rm hol}^{\mathcal H}_\gamma$ is a group morphism and, since it is also a diffeomorphism, it is a Lie group isomorphism.
\end{proof}

\begin{theorem}\label{th:LT}
If $\pi\colon \mathfrak{G}\to M$ admits a group connection then it is a locally trivial group bundle. 
\end{theorem}

\begin{proof}
Let $p$ be any point of $M$. Let us consider an open neighborhood $\mathcal U$ of $p$ endowed with coordinates $\lambda_1,\ldots,\lambda_n$ vanishing at $p$ and representing $\mathcal U$ as the open unit cube  $\left(-\frac{1}{2},\frac{1}{2}\right)^n$. Let us consider the vector fields in $\mathcal U$, $X_i = \frac{\partial}{\partial\lambda_i}$, and its horizontal liftings $\tilde X_i$ to $\mathfrak{G}|_{\mathcal U}$. Note that the diffeomorphism given by the following formula:
$$\psi\colon \mathfrak{G}\mid_{\mathcal{U}} \, \to  \mathfrak{G}_p\times \mathcal U,\quad g \mapsto ({\rm exp}(-\lambda_n(p)\tilde X_n)\cdots   {\rm exp}(-\lambda_1(p)\tilde X_1)g ,\pi(g))$$
is, by Remark \ref{rm:vector_holonomy}, compatible with the group operation and therefore
a smooth group bundle trivialization of $\mathfrak{G}|_{\mathcal U}$.
\end{proof}

\subsection{On the space of group connections} Let us examine the exact sequence \eqref{eq:se_jet}. For $x\in M$, the tangent space to $\mathfrak{G}$ at $\mathfrak e_x$ has a canonical decomposition,
$$T_{\mathfrak e_x}\mathfrak{G} = T_{\mathfrak e_x}(\mathfrak e(M))\oplus {\rm Lie}(\mathfrak{G}_x).$$
The isomorphism $\mathrm{d}_x\mathfrak e$ between $T_xM$ and $T_{\mathfrak e_x}(\mathfrak e(M))$ induces a canonical isomorphism,
$$\iota\colon K^1\mathfrak{G}\xrightarrow{\sim} {\rm Lin}(TM,{\rm Lie}(\mathfrak{G})),\quad j_x^1g \mapsto \mathrm{d}_xg - \mathrm{d}_x\mathfrak e.$$
Such isomorphism is a group bundle isomorphism. It follows that $K^1\mathfrak G\to M$ is a vector bundle, its group operation is commutative. For sections $g$ and $\tilde g$ of $\mathfrak{G}$ with $g(x) =\tilde g(x) = \mathfrak e_x$ we have:
$$j^1_xg\cdot j_x^1\tilde g = j^1_x\tilde g\cdot j_x^1g$$
Sections of $K^1\mathfrak{G}\to M$ can be seen as $1$-forms in $M$ with values in ${\rm Lie}(\mathfrak G)$.

\begin{remark}
It follows that 
$\pi_{1,0}\colon J^1\mathfrak{G}\to \mathfrak{G}$ is endowed of a canonical structure of affine bundle modeled over the pullback to $\mathfrak G$ of the vector bundle ${\rm Lin}(TM,{\rm Lie}(\mathfrak{G}))\to M.$
Given $\varphi_x \in {\rm Lin}(T_xM,{\rm Lie}(\mathfrak{G}_x))$ and $j_x^1\in J^1\mathfrak{G}$ we define:
$$j^1_xg + \varphi_x = \iota^{-1}(\varphi_x)\cdot j^1_xg,$$
simmilarly given $j^1_xg$ and $j^1_x\tilde g \in J^1\mathfrak G$ with $g(x) = \tilde g(x)$ we define the difference:
$$j^1_x\tilde g -  j^1_x g = \iota(j^1_x\tilde g \cdot j^1_x g^{-1})$$
\end{remark}

\begin{lemma}\label{lm:arduo}
Let $j^1_xh\in J^1\mathfrak{G}$ and $j^1_xg\in K^1\mathfrak G$, then $$\iota(j^1_xh\cdot j^1_xg \cdot j^1_xh^{-1}) = {\rm Adj}_{h(x)}\circ \iota(j^1_xg).$$
\end{lemma}

\begin{proof} Let us consider $\Phi_h\colon \mathfrak{G}\to \mathfrak{G}$ the fibre preserving map that gives the internal automorphism $L_{h(y)}\circ R_{h(y)}^{-1}$ of $\mathfrak{G}_y$ for $y$ in a neighborhood of $x$. Then, we have:
$$\iota(j^1_xh\cdot j^1_xg \cdot j^1_xh^{-1}) = \mathrm{d}_x(h\cdot g \cdot h^{-1}) - \mathrm{d}_x{\mathfrak e} =$$
$$\mathrm{d}\Phi_h \circ \mathrm{d}L_g \circ \mathrm{d}_x\mathfrak e - \mathrm{d}_x\mathfrak e = \mathrm{d}\Phi_h\circ(\mathrm{d}L_g\circ \mathrm{d}_x\mathfrak e - \mathrm{d}_x\mathfrak e)=$$ 
$$\mathrm{d}\Phi_h(\mathrm{d}_xg - \mathrm{d}_x\mathfrak e) =
{\rm Adj}_{h(x)}\circ \iota(j^1_xg).$$
\end{proof}

Let us consider the following group bundle
${\rm Lin}(TM,{\rm Lie}(\mathfrak{G}))\rtimes \mathfrak{G}\to M$
defined as the semidirect product:
$$(\varphi, g)\cdot(\psi, h) = (\varphi + {\rm Adj}_{g}\circ \psi, g\cdot h).$$ 
This group bundle is isomorphic to $J^1\mathfrak{G}$, but not in a canonical way. 
Let us see that group connections in $\mathfrak G\to M$ are equivalent to decompositions of $J^1\mathfrak{G}$
as such semidirect product.

\begin{proposition}
Let $\mathcal H$ be a group connection in $\pi\colon \mathfrak G\to M$, and $\sigma$ its induced section of $\pi_{1,0}\colon J^1\mathfrak{G}\to \mathfrak{G}$. The following map,
$$\Psi \colon J^1\mathfrak{G} \xrightarrow{\,\sim\,} {\rm Lin}(TM,{\rm Lie}(\mathfrak{G})) \rtimes \mathfrak{G}, \quad j_x^1g \mapsto (j^1_xg\cdot \iota(\sigma(g(x))^{-1}), g(x))$$
is a group bundle isomorphism. Reciprocally, any such isomorphism defines a group connection by setting
$\sigma(g) = \Psi^{-1}(0,g)$ for all $g\in G$.
\end{proposition}

\begin{proof}
For each $x\in M$ we have $\sigma(\mathfrak{G}_x) \cap K^1\mathfrak{G}_x = \{\mathfrak e_x\}$. Therefore, we have a semidirect product decomposition 
$$(J^1\mathfrak{G})_x \xrightarrow{\sim} (K^1\mathfrak{G})_x\rtimes \mathfrak{G}_x,\quad 
j_x^1g\mapsto (j_x^1g\cdot \sigma(g(x))^{-1},g(x)).$$ 
Finally we compose with $\iota$ and, by Lemma \ref{lm:arduo}, we obtain the required Lie group isomorphism.
\end{proof}

Let us recall that if $G$ is a Lie group and $\rho\colon G\to {\rm End}(V)$ is a representation then a $1$-cocycle of $G$ with values in $V$ is a smooth map $\gamma\colon G\to V$ satisfying:
$$\gamma(g\cdot h) = \rho(g)(\gamma(h)) + \gamma(g).$$
The set $\mathcal Z^1(G,V)$ of $1$-cocycles is a vector space. In particular the space of smooth maps from $G$ to ${\rm Lie}(G)$ is identified with $\mathfrak X(G)$ through the so called right-logarithmic derivative:
$$\mathfrak X(G) \mapsto \mathcal C^\infty(G,{\rm Lie}(G)), \quad X \to \ell_d(X), \quad \ell_d(X) (g) = dR_{g}^{-1}(X_g).$$
The following, is a well known characterization of cocycles of a Lie group with values in its Lie algebra. The proof is elementary and left to reader.

\begin{proposition}\label{lm:cocycles}
Let $G$ be a Lie group; a vector field $X\colon G\to TG$ is a Lie group morphism if and only if $\ell_d(X)$ is a $1$-cocycle of $G$ with values in ${\rm Lie}(G)$. That is, $\gamma = \ell_d(X)$ satisfies:
$$\gamma(g\cdot h) = {\rm Adj}_g(\gamma(h)) + \gamma(g).$$
\end{proposition}

Let us consider $\sigma$ and $\tilde\sigma$ two group connections on $\mathfrak G\to M$. The affine bundle structure of $J^1\mathfrak{G}\to \mathfrak{G}$ allows us to define their difference $\theta = \sigma - \tilde\sigma$ in the following way:
$$\theta \colon \mathfrak{G} \to {\rm Lin}(TM,{\rm Lie}(\mathfrak{G})),
\quad g \mapsto \theta(g):= \sigma(g) - \tilde\sigma(g) = \iota(\sigma(g)\tilde\sigma(g)^{-1})
.$$

\begin{lemma}\label{lm:diff}
Let $\mathcal H$ and $\tilde{\mathcal H}$ be two group conections with induced sections $\sigma$ and $\tilde\sigma$ respectively. The difference
$\theta = \sigma - \tilde \sigma$ is $1$-form in $M$ with values in the space of $1$-cocycles $Z^{1}(\mathfrak{G},{\rm Lie}(\mathfrak{G}))$. That is, for each $\vec X\in T_xM$ the map:
$$i_{\vec X}\theta \colon \mathfrak{G}_x \to {\rm Lie}(\mathfrak{G}_x),\quad g\mapsto \theta(g)(\vec X)$$
is a $1$-cocycle. 
\end{lemma}

\begin{proof}
Let us consider $g,h\in \mathfrak{G}$ with $\pi(g) = \pi(h) = x$ and $\vec X\in T_xM$. 
In order to simplfy the calculation we identify $K^1\mathfrak G$ with ${\rm Lin}(TM,{\rm Lie}(\mathfrak G))$ by means of the isomorphism $\iota$.

\begin{align*}
(i_{\vec X}\theta)(gh) &= (\sigma(gh)\tilde\sigma(gh)^{-1})(\vec X) \\
&=
\sigma(g)\tilde\sigma(g)^{-1}(\tilde\sigma(gh)\sigma(h)^{-1}\tilde\sigma(g)^{-1})^{-1}  (\vec X) \\ 
  &=  (\sigma(g) - \tilde \sigma(g)  - {\rm Adj}_g(\tilde\sigma(h) - \sigma(h) ) )(\vec X) \\
  &=
(\theta(g) + {\rm Adj}_g(\theta(h)))(\vec X) \\
  &= (i_{\vec X}\theta)(g) + {\rm Adj}_g((i_{\vec X}\theta)(h)).
\end{align*}
\end{proof}

From now on, given two group connections $\sigma$ and $\tilde\sigma$ we refer to the $1$-form $\theta$ defined in Lemma
\ref{lm:diff} as its difference $\theta = \sigma - \tilde\sigma$.

\begin{theorem}\label{th:GC_AS}
Let $\pi\colon \mathfrak{G}\to M$ be a locally trivial group bundle. Its space of group connections is an affine space modeled over the vector space $\Omega^1(M,\mathcal Z^1(\mathfrak{G},{\rm Lie}(\mathfrak{G}))$. 
\end{theorem}

\begin{proof}
As we have seen in Lemma \ref{lm:diff}, the difference between the induced sections corresponding to two group connections is an element of $\Omega_1(M,\mathcal Z^1(\mathfrak{G},{\rm Lie}(\mathfrak{G}))$. Reciprocally, given a group connection with induced section $\sigma$ and $\theta\in \Omega_1(M,\mathcal Z^1(\mathfrak{G},{\rm Lie}(\mathfrak{G}))$, it follows from Propositon \ref{lm:cocycles} that $\sigma+\theta$ is a group bundle morphism and therefore it is the section corresponding to a group connection. From an standard argument on affine connections we obtain that the group connections in a locally trivial group bundle are in direct correspondence with the sections of an affine bundle and therefore, we also obtain that the space of group connections is not empty. 
\end{proof}

\section{Fibre bundles}\label{s_FB}

\subsection{Principal and associated bundles} Let us recall that a principal $G$-bundle\footnote{Here $G$ stands for a Lie group.} $\pi\colon P\to M$ is a smooth bundle endowed with a smooth right action $P\times G\to P$, $(p,g)\mapsto p\cdot g$ which is free and transitive on each fiber of $\pi$. An important concept we are going to use extensively is that of \emph{balanced construction} or \emph{associated bundle}. Let us consider $\pi\colon P\to M$ a principal $G$-bundle and $V$ a $G$-manifold, that is, a smooth manifold endowed with a left action $G\times V\to V$. In the direct product $P\times V$ we consider the action of $g$ given by,
$$(P\times V)\times G\to P\times V \to ((p,v),g)\mapsto (p\cdot g, g^{-1}v).$$
The space of orbits of such action is called the \emph{balanced construction} and denoted by $P\times_G V$. It admits a unique smooth structure such that the identification map,
$$\mathfrak q\colon P\times V \to P\times_G V,\quad (p,v) \mapsto [p,v]$$
is a submersion, and moreover, a principal $G$-bundle.\footnote{The name balanced construction comes from the fact that $G$-orbits in $P\times V$ are the equivalence class under the equivalence relation,
$\forall g\in G\,\forall p\in P\, \forall v\in V\,\,\,\,(p\cdot g,v) \sim (p, g\cdot v)$.}
It is endowed of a projection,
$$\bar\pi \colon P\times_GV \to M, \quad [p,v]\mapsto \pi(p)$$
that realizes it as bundle over $M$ with fiber $V$. Each element $p\in P$ induces an isomorphism\footnote{Because of this, in contexts in which the main geometric object is the balanced construction $E$, elements of $P$ are termed \emph{frames}.} between the fiber $(P\times_G V)_{\pi(p)}$ and $V$,
$$\hat{p}\colon V \xrightarrow{\sim} (P\times_G V)_{\pi(p)}, \quad v\mapsto [p,v].$$
Whenever the action of $G$ in $V$ is not clear by context, we will write $P\times_\alpha V$ to denote the bundle associated to the action $\alpha$ of $G$ in $V$.\\

An important feature of the associated bundles is that they are compatible with additional structure in $V$, as long as it is $G$-invariant. Let us assume that $V$ is endowed with some additional structure, namely, a vector space, Lie algebra, finite algebra, Lie group, etc, such that $G$ acts on $V$ by automorphisms. Then, we may push forward this structure to the fibers of $P\times_G V$ at any point by means of any frame. We will obtain that $\bar\pi\colon P\times_G V\to M$ is a bundle by vector spaces, Lie algebras, finite algebras, Lie groups, etc. In our applications, we will make use of the following, well known, examples.

\begin{example}
For each fiber $P_x$ let us consider ${\rm Aut}_G(P_x)\subset {\rm Diff}(P_x)$ the group of $G$-equivariant automorphisms of the $P_x$. Let us construct an associated bundle whose fiber at $x\in M$ is a group canonically isomorphic to ${\rm Aut}_G(P_x)$. Let us consider the adjoint action ${\rm Adj}\colon G\times G\to G$. It is an action by group automorphisms. The associated bundle
$P\times_{\rm Adj} G$ is then a group bundle over $M$. We call it the \emph{gauge bundle} of $P$,
$${\rm Gau}(P) = P\times_{\rm Adj} G\to M.$$
Now, for each $x\in M$ we have a natural isomorphism of groups,
$${\rm Aut}_G(P_x) \xrightarrow{\sim} {\rm Gau}(P)_x, \quad \sigma \to [p,\sigma(p) ] = \{(p,\sigma(p)) \,\colon\, p\in P_x\}.$$
Thus, ${\rm Gau}(P)\to M$ is a locally trivial group bundle of fiber $G$.
\end{example}

\begin{example}\label{ex:IsoP}
 We also may consider in $P$ the canonical left action defined by duality $g\cdot p := p\cdot g^{-1}$. Thus we may construct the balanced construction ${\rm Iso}(P) = P\times_G P$ but it is endowed with two projections onto $M$ $s[p,q] = \pi(p)$ and $t[p,q] = \pi(q)$. This ${\rm Iso}(P)\rightrightarrows M$ is endowed then of a groupoid structure:
 $$\mathfrak m\colon {\rm Iso}(P){}_s\hspace{-1mm}\times_t {\rm Iso}(P)\to {\rm Iso}(P), \quad ([q,r],[p,q] )\mapsto [p,r]$$
 $$\mathfrak i \colon {\rm Iso}(P)\to {\rm Iso}(P),\quad [p,q] \mapsto [q,p]$$
 $$\mathfrak e\colon M \to {\rm Iso}(P), x \to [p,p], \mbox{ where } \pi(p) = x.$$
 For each pair of fibers $P_x$, $P_y$ let us consider ${\rm Iso}_G(P_x,P_y)\subset {\mathcal C}^{\infty}(P_x,P_y)$ the space of  $G$-equivariant isomorphisms from $P_x$ to $P_y$. There are canonical isomorphisms,
 $${\rm Iso}_G(P_x,P_y) \xrightarrow{\sim} {\rm Iso}(P)_{xy}, \quad \sigma \mapsto [p,\sigma(p)]\mbox{ where } \pi(p) = x,$$
 compatible with composition and inversion. The grupoid ${\rm Iso}(P)\rightrightarrows M$ is thus called the groupoid of gauge isomorphisms of $P$ over $M$. It is, by definition, a transitive Lie groupoid. The stabilizer group bundle $\rm eq(s,t)\to M$ is canonically isomorphic to ${\rm Gau}(P)$,
 $${\rm Gau}(P) \hookrightarrow {\rm Iso}(P), \quad [p,g] \mapsto [p,p\cdot g].$$
\end{example}

\begin{example}
For each fiber $P_x$ let us consider ${\rm aut}_G(P_x)\subset \mathfrak X(P_x)$ the Lie algebra of $G$-invariant vector fields on $P_x$. As the infinitesimal version of the above example, we may construct an associated bundle whose fiber at $x\in M$ is a Lie algebra canonically isomorphic to ${\rm aut}_G(P_x)$. Let us consider the adjoint action ${\rm adj}\colon G\times {\rm Lie}(G)\to {\rm Lie}(G)$. It is an action by Lie algebra isomorphisms. The associated bundle
$P\times_{\rm adj} {\rm Lie}(G)$ is a bundle by Lie algebras over $M$. We call it the \emph{infinitesimal gauge bundle} of $P$,
$${\rm gau}(P) = P\times_{\rm adj} {\rm Lie}(G)\to M.$$
Now, for each $x\in M$ we have a canonical isomorphism of Lie algebras,
$${\rm gau}(P)_x \xrightarrow{\sim} {\rm aut}_G(P_x), \quad [p,A]\mapsto B$$
where 
$$B_{p\cdot g} = \left.\frac{\mathrm{d}}{\mathrm{d}t}\right|_{t=0}(p\cdot {\exp}(tA) \cdot g).$$
\end{example}

\begin{example}\label{ex:vector_associated}
Let $\pi\colon E\to M$ be a vector bundle of rank $n$. Let us consider $L(E)\to M$ the bundle of linear frames on $E$, that is, for each $x\in M$, $L(E)_x$ is the set of linear bases of the vector space $E_x$. As $L(E)$ is an open subset of $E\times_M \ldots \times_M E$ it is a smooth manifold. There is a natural action of ${\rm GL}_n$ on $E$ by linear changes of basis.
$$((e_1,\ldots,e_n),(\sigma_{ij}))\mapsto \left(\sum e_i\sigma_{i1},\ldots, \sum e_i\sigma_{in}\right)$$
that gives $L(E)$ the structure of a ${\rm GL}_n$-principal bundle. Any linear representation of ${\rm GL}_n$ produces an associated vector bundle. In particular, the tautological representation\footnote{${\rm GL}_n$ acting in $\mathbb R^n$  by matrix composition, here elements of $\mathbb R^n$ as thought as column vectors.} recovers the vector bundle,
$$L(E)\times_{{\rm GL}_n} \mathbb R^n\xrightarrow{\sim}E, \quad
[(e_1,\ldots,e_n),(\lambda_1,\ldots,\lambda_n)]\mapsto \sum \lambda_ie_i.
$$
\end{example}

\begin{example}\label{ex:Aut_aut}
Let $\pi\colon P\to M$ a principal $G$-bundle, $V$ a $G$-manifold with action $\alpha$ and 
$\bar\pi\colon P\times_\alpha V = E\to M$ its corresponding associated bundle. Let us consider the identification map $\mathfrak q\colon P\times V \to E$. Let us fix $x\in M$ and $\mathfrak q_x = \mathfrak q|_{P_x\times V}$.  
Let us note that any $G$-equivariant automorphism $\sigma$ of $P_x$ factorizes through $\mathfrak q_x$ and induces a diffeomorphism $\bar\sigma$ of $E_x$, $\bar\sigma([p,v])= [\sigma(p),v]$. We obtain a group morphism ${\rm Gau}(P)_x \to {\rm Diff}(E_x)$ whose kernel is conjugated with ${\rm ker}(\alpha)$. By considering the image of these morphisms as $x$ varies in $M$ we obtain a quotient group bundle,
$$\alpha_*\colon {\rm Gau}(P)\to {\rm Aut}(E)$$
where ${\rm Aut}(E)\to M$ is a locally trivial group bundle of fiber $G/{\rm ker}(\alpha)$ and ${\rm Aut}(E)_x$ is a finite dimensional Lie subgroup of ${\rm Diff}(E_x)$ for all $x\in M$. 
By derivation of the above morphism we obtain a Lie algebra bundle morphism,
$$\alpha_*'\colon {\rm gau}(P) \to {\rm aut}(E)$$
where for all $x\in M$ ${\rm aut}(E_x)$ is a finite dimensional Lie algebra of vector fields in $E_x$. In particular if $X\in {\rm gau}(P_x)$ then $(X,0)$ is a $\mathfrak q_x$-projectable vector field in $P_x\times V$ and $\alpha_*'(X) = (\mathfrak q_x)_*(X,0)$.
Note that for faithful actions the above morphisms $\alpha_*$ and $\alpha_*'$ are isomorphims. 
\end{example}

\begin{example}
For the particular case of a vector bundle $E\to M$ we have ${\rm Aut}(E)\to M$ the group bundle of linear automorphisms of the fibers and ${\rm aut}(E)$ the Lie algebra bundle of vertical linear vector fields.  
\end{example}

\subsection{Principal and associated connections}

\begin{definition}
Let $G$ be a Lie group with Lie algebra $\mathfrak{g}$. Let $P\to M$ be a principal $G$-bundle. A {\em{principal connection}} on $P$ is a connection $\mathcal{H}$ on $P$ that is invariant by the action of $G$, i.e., for all $p\in P$ and all $g\in G$
\[\mathrm{d}_p \gamma_g(\mathcal{H}_p) = \mathcal{H}_{p\cdot g},
\] where $\gamma_g:P\to P$ denotes the diffeomorphism given by the action of $g$ on $P$. 
\end{definition}

Note that the smooth map $\beta_p:G\to P_x$ given by the action on the element $p$ induce an isomorphism between the vertical space $\mathrm{Ver}_p(P)$ and the Lie algebra $\mathfrak g$ which allows us to canonically associate to the distribution $\mathcal{H}$ a $\mathfrak g$-valued $1$-form $\omega$ on $P$ such that $\mathrm{Ker}(\omega_p)=\mathcal{H}_p$,
for all $p\in P$. 
The $G$-invariance of the connection $\mathcal{H}$
is equivalent to the $\mathrm{Ad}$-equivariance of $\omega$, that is 
\begin{equation}\label{eq:omegaequivariant}
\gamma_g^*\,\omega=\mathrm{Adj}_{g^{-1}}\circ\omega,
\end{equation}
for all $g\in G$. 
Therefore the covariant derivative operator is defined by $\nabla: \Gamma(P,M) \to \Omega^1(M;\mathfrak{g}), \epsilon \mapsto \epsilon^*{\omega}$ and the $\mathrm{Ad}$-equivariance condition means that $\nabla (\epsilon \cdot g) = \mathrm{Adj}_{g^{-1}} \circ  \nabla \epsilon$, for all $g\in G$ and all $\epsilon \in \Gamma(P,M)$.

\begin{definition}
Let $\pi:P\to M$ be a $G$-principal bundle and let $N$ be a manifold endowed with a smooth left action $G\times N\to N$. Consider the associated bundle $\bar\pi:P\times_G N \to M$ and denote by $\mathfrak{q}:P\times N\to P\times_G N$ the identification map. Given a principal connection $\mathcal{H}$ on $P$ the formula:
\[
\bar{\mathcal{H}}_{[p,n]} = \mathrm{d}_{(p,n)} \mathfrak{q}(\mathcal{H}_p\oplus \{0\}), \, p\in P, n\in N,
\]
defines 
a connection $\bar{\mathcal{H}}$ on $P\times_G N\to M$, 
called the \emph{associated connection} induced by $\mathcal{H}$. 
\end{definition}

\begin{example}
Let $\rho\colon E\to M$ be a rank $n$ vector bundle. It is canonically identified with the associated bundle corresponding to principal bundle $L(E)\to M$ and the canonical action of ${\rm GL}_n$ on $\mathbb R^n$ (example \ref{ex:vector_associated}). Associated connections in $E$ induced by principal connections in $L(E)$ are termed \emph{linear connections}. They admit the following equivalent characterizations: 
\begin{itemize}
\item[(a)] Let $\pi\colon E\to M$ be a vector bundle. Then $\pi_{1,-1}\colon J^1E\to M$ is also a vector bundle with the operations:
$$j_x^1e + j_x^1 v = j_x^1(e+v), \quad \lambda(j_x^1)v = j_x^1(\lambda v),$$
and the projection $\pi_{1,0}\colon J^1E \to E$ is a surjective morphism of vector bundles over $M$. A connection in $E$ is linear if and only if the induced section $\sigma$ of $\pi_{1,0}$ is linear on fibers.
\item[(b)]
Let $\pi:E\to M$ be a vector bundle endowed with a connection. For each $e\in E$ the vertical space is given by $\mathrm{Ver}_eE = E_{\pi(e)}$ then $\mathrm{Ver}(E) = E\times_M E$. If $\epsilon \in \Gamma(M,E)$ is a smooth section of $\pi$, clearly $\epsilon^*(\mathrm{Ver}(E))=E$ and therefore the covariant derivative operator is $\nabla :\Gamma(M,E)\to \Omega^1(M,\Gamma(M,E))$. That is given a section $\epsilon$ and a vector field $X$ the value of the covariant derivative $(\nabla_X \epsilon)(x)$ is a vector field in $E_x$ that we can see as a smooth map from $E_x$ onto itself.
An Ehreshmann connection on $E$ is linear if and only if its covariant derivative operator $\nabla$ satisfies the condition
\[
\nabla(f\epsilon +\rho) = \mathrm{d}f \otimes \epsilon +  f\nabla \epsilon + \nabla \rho, 
\]for all $\epsilon, \rho \in \Gamma(M,E)$, and all $f\in C^{\infty}(M)$.
\end{itemize}
\end{example}

\begin{example}
Let $\pi\colon E\to M$ be a vector bundle. Then there is canonical embedding ${\rm End}(E) \subset \mathfrak X_{{\rm Ver}(E)}$. Any endomorphism $A$ of $E_x$ defines a linear vector field, 
$$X_v = \left.\frac{\mathrm{d}}{\mathrm{d}t}\right|_{t=0}(v+ tA(v))$$
in $E_x$. If a connection is linear then its curvature $2$-form (Definition \ref{df:curvature}) takes values in this subbundle and therefore $R\in \Omega^2(M,{\rm End}(E))$. It corresponds with the classical notion of curvature of a linear connection as in \cite[p. 133]{kobayashi}.
\end{example}

\begin{example}\label{ex:curvatureP}
Let $\pi\colon P\to M$ be a $G$-principal bundle. For each fiber $P_x$ we have the Lie algebra ${\rm aut}(P_x)$ of $G$-invariant vector fields, which is a Lie algebra isomorphic to ${\rm Lie}(G)$. We may construct a vector bundle ${\rm gau}(P)\to M$ which is, in fact, the associated bundle to the adjoint action of $G$ on ${\rm Lie}(G)$. If $\mathcal H$ is a principal connection is principal then its curvature $2$-form takes values in this finite dimensional vector bundle
${\rm gau}(P)\subset \mathfrak X_{\rm Ver(P)}$. We have $R\in\Omega^2(M,{\rm gau}(P))$ and therefore its pullback $\pi^*(R) \in \Omega^2(P,{\rm Lie}(G))$ its a ${\rm adj}$-equivariant $2$-form. This pullback is the usual curvature form, as in \cite[p. 77]{kobayashi}. 
\end{example}


\begin{proposition}\label{pr:curvatures_as}
Let $\pi\colon P\to M$ be a $G$-principal bundle, $V$ a $G$-manifold with action $\alpha$ and
$\bar\pi\colon P\times_\alpha V = E\to M$ its associated bundle. Let us consider $\mathcal H$ a principal connection in $P$ with curvature tensor $R$ and $\bar{\mathcal H}$ its corresponding associate connection in $E$ with curvature tensor $\bar R$.
\begin{itemize}
    \item[(a)] For any $X,Y\in T_xM$ we have $\bar R(X,Y)\in {\rm aut}(E_x)$ and therefore \linebreak $\bar R\in \Omega^2(M, {\rm aut}(E))$.
    \item[(b)] Let us consider the canonical projection $\alpha_*'\colon {\rm gau}(P) \to {\rm aut}(E)$ (example \ref{ex:Aut_aut}). For any 
    $X,Y\in T_xM$ we have $\bar R(X,Y) = \mathfrak \alpha_*'R(X,Y)$.
\end{itemize}
\end{proposition}

\begin{proof}
Let us consider $X$ and $Y$ vector fields in $M$, $\bar X$ and $\bar Y$ their $\bar{\mathcal H}$-horizontal lifts to $E$, $\tilde X$ and $\tilde Y$ their $\mathcal H$-horizontal lifts to $P$ and $x\in M$. Note that the vector fields $(\tilde X,0)$ and $(\tilde Y,0)$ in $P\times V$ are $\mathfrak q$-projectable and they project onto $\bar X$ and $\bar Y$, and similarly for their Lie bracket. Therefore we have:
$$\mathfrak q_*(2R(X,Y),0) = \mathfrak q_*([\tilde X,\tilde Y] - \widetilde{[X,Y]},0) = [\bar X, \bar Y] - \overline{[X,Y]} = 2\bar R(X,Y).$$
Then, for any $x\in M$ we have:
$$\alpha'_*(R(X_x,Y_x)) = (\mathfrak q_x)_*(R(X_x,Y_x),0) = \bar R(X_x,Y_x).$$
\end{proof}

\begin{proposition}\label{pr:hol_as}
Let $\pi\colon P\to M$ be a $G$-principal bundle, $V$ a $G$-manifold and
$\bar\pi\colon P\times_G V = E\to M$ its associated bundle. Let us consider $\mathcal H$ a principal connection in $P$ with curvature tensor $R$ and $\bar{\mathcal H}$ its corresponding associate connection in $E$ with curvature tensor $\bar R$. For any $\gamma\in \ell(M,x,x)$ the following hold:
\begin{itemize}
    \item[(a)]  ${\rm hol}_{\gamma}^{\mathcal H}\in {\rm Gau}(P)_x$;
    \item[(b)]  ${\rm hol}_{\gamma}^{\bar{\mathcal H}}\in {\rm Aut}(E)_x$;
    \item[(c)]  ${\rm hol}_{\gamma}^{\bar{\mathcal H}} = \alpha_*({\rm hol}_{\gamma}^{\mathcal H})$.
\end{itemize}
\end{proposition}

\begin{proof}
(a) 
Let $\tilde\gamma$ be the $\mathcal H$-horizontal lift of $\gamma$ with starting point $p\in P_x$. Then for each $g\in G$, $R_g\circ \tilde\gamma$ is the $\mathcal H$-horizontal lift of $\gamma$ starting at $p\cdot g$. Therefore:
$${\rm hol}_{\gamma}^{\mathcal H}(p\cdot g) = R_g\circ \tilde \gamma (1) = {\rm hol}_{\gamma}^{\mathcal H}(p) \cdot g.$$
(b) and (c) follow from (a) and the fact that $\mathfrak q$ sends $(\tilde\gamma,v)$ to the $\bar{\mathcal H}$-horizontal lift of $\gamma$ with starting point $[\tilde\gamma(0),v]$.  
\end{proof}

The classical theorem of Ambrose-Singer it is usually formulated in terms of the curvature tensor seen as a basic $2$-form in $P$ with values in ${\rm Lie}(G)$ and the holonomy bundle. However, it is easily reformulated in terms of the curvature form $R$ and the with values in ${\rm gau}(P)$. 

\begin{theorem}[Ambrose-Singer \cite{ambrose1953theorem}]\label{th:AS}
${\rm Lie}({\rm RHol}_x^{\mathcal H})$ is the Lie subalgebra of ${\rm gau}(P)_x$ generated by elements ${\rm hol}_{\gamma*}^{\mathcal H}(R(X,Y))$ where $\gamma$ varies along piece-wise smooth paths with endpoint $x$; $X$, $Y$ vary along vectors with $T_{\gamma(0)}M$ and ${\rm hol}_{\gamma*}^{\mathcal H}\colon \mathfrak X(P_{\gamma(0)}) \to \mathfrak X(P_x)$ is the Lie algebra morphism (that maps ${\rm gau}(P_{\gamma(0)})$ isomorphically into ${\rm gau}(P_x)$) induced by ${\rm hol}_{\gamma}$.
\end{theorem}

From Propositions \ref{pr:curvatures_as} and \ref{pr:hol_as} it clearly follows that the Ambrose-Singer theorem also holds for associated connections. 

\begin{corollary}\label{cr:AS_associated}
${\rm Lie}({\rm RHol}_x^{\bar{\mathcal H}})$ is the Lie subalgebra of ${\rm aut}(E_x)$ generated by elements ${\rm hol}_{\gamma*}^{\bar{\mathcal H}}(\bar R(X,Y))$ where $\gamma$ varies along piece-wise smooth paths with endpoint $x$; $X$, $Y$ vary along vectors with $T_{\gamma(0)}M$ and ${\rm hol}_{\gamma*}^{\bar{\mathcal H}}\colon \mathfrak X(E_{\gamma(0)}) \to \mathfrak X(E_x)$ is the Lie algebra morphism (that maps ${\rm aut}(E_{\gamma(0)})$ isomorphically into ${\rm aut}(E_x)$) induced by ${\rm hol}_{\gamma}^{\bar{\mathcal H}}$.
\end{corollary}

\begin{proof}
By proposition \ref{pr:hol_as} we have thar ${\rm RHol}_x^{\bar{\mathcal H}} = \alpha_*({\rm RHol}_x^{\mathcal H})$. We conclude by Proposition \ref{pr:curvatures_as} (b) and Ambrose-Singer (Theorem \ref{th:AS}).
\end{proof}

\section{Transitive Lie groupoids and principal bundles}\label{s_TSLG}

There is a natural correspondence between transitive Lie groupoids and principal bundles, implicitly exposed in \cite{mackenzie1987lie}. In this section we show that this correspondence extends to a correspondence between representations of a transitive Lie groupoids and associated bundles. Moreover, this correspondence relates some class of connections -that we call infinitesimally generated connections- to associated connections. Our purpose is to show that group bundles are associated bundles and group connections are associated connections. From now on let $\mathcal G\rightrightarrows M$ be a transitive Lie groupoid.

\subsection{Equivalence of transitive Lie grupoids and principal bundles.} \label{eqv_cat}
It is well known that for any $x\in M$ the target projection $t\colon \mathcal G_{x\bullet}\to M$ is a principal bundle with structure group $\mathcal G_{xx}$ acting by right composition.  Moreover, all these principal bundles are isomorphic, any element $g\in \mathcal G_{xy}$ induces a pair of isomorphisms:
$$\varphi_g\colon \mathcal G_{yy}\to \mathcal G_{xx}, \quad  h \mapsto g^{-1}hg,$$
$$R_g\colon \mathcal G_{y\bullet} \to \mathcal G_{x\bullet}, \quad \sigma \mapsto \sigma g,$$
where $\varphi_g$ is a Lie group isomorphism a $R_g$ is a principal bundle isomorphism defined over $\varphi_g$ in the sense that $R_g(p\circ h) = R_g(p) \circ \varphi_g(h)$. Therefore, the fibers of $s$ give a family of isomorphic principal bundles modeled over isomorphic groups. On the other hand, let us consider the construction given in Example \ref{ex:IsoP} that assigns to each principal bundle $P\to M$ its groupoid of isomorphisms ${\rm Iso}(P)\rightrightarrows M$. These assignations are reversible in the following sense:

\begin{itemize}
    \item[(a)] Let $\pi\colon P\to M$ a principal bundle modeled over $G$. Each point $p\in P$ with $x = \pi(p)$ induces isomorphisms of groups and bundles:
    $$\phi_p\colon  G \to {\rm Aut}_G(P_x) = {\rm Iso}(P)_{xx}, \quad   g \mapsto [p,pg]$$
    $$R_p\colon P \to {\rm Iso}(P)_{x\bullet}, \quad q \mapsto [p,q]$$
    so that $R_p(q\cdot g) = R_p(q)\phi_p(g)$ for all $q\in P$ and $g\in G$.
    \item[(b)] Let $\mathcal G\rightrightarrows M$ be a transitive Lie groupoid. Each element  $g\in \mathcal G$ with source $s(g) = x$ induces a Lie groupoid isomorphism,
    $$\mathcal G \to {\rm Iso}(\mathcal G_{x\bullet}), \quad h \mapsto L_h|_{\mathcal G_{xs(h)}}$$
    where $L_h\colon \mathcal G_{\bullet s(h)} \to \mathcal G_{\bullet t(h)}$, $\sigma \mapsto h\sigma $ is the left translation. 
    \end{itemize}

The above correspondence can be seen an equivalence of categories. We may set the following categories:
\begin{itemize}
    \item[(i)] ${\bf Ppal}$ is the category of principal bundles over $M$. Objects are $4$-tuples $(P,G,\pi,\alpha)$ where $\pi\colon P\to M$ is a fiber bundle, $G$ is a Lie group, and $\alpha$ is a right action of $G$ in $P$ that gives the latter an structure of principal bundles. A morphism from $(P,G,\pi,\alpha)$ to $(P',G',\pi',\alpha')$ is a
    from are pair of maps,
    $$f\colon P\to P', \quad \sigma \colon G,\to G'$$
    such $f$ is a morhphisms of fiber bundles over $M$, $\sigma$ is a Lie group morphism and $f(p\cdot g) = f(p)\cdot \sigma(g)$.
    \item[(ii)] ${\bf Grpd}$ is the category of transitive Lie groupoids over $M$. Objects are transitive Lie groupoids $\mathcal G\rightrightarrows M$ and morphisms are morphisms of Lie groupoids that induce the identity on $M$.
\end{itemize}

In this way we may set the natural transformations,
$${\rm Iso}\colon {\bf Ppal} \leadsto {\bf Grpd}, \quad P\to M \,\, \leadsto \,\, {\rm Iso}(P)\rightrightarrows M,$$
$${\rm sou}_x\colon {\bf Grpd} \leadsto {\bf Ppal}, \quad \mathcal G\rightrightarrows M \,\, \leadsto \,\, 
(\mathcal G_{x\bullet},\mathcal G_{xx}, t|_{\mathcal G_{x\bullet}}, \circ)$$
for each $x\in M$. The isomorphisms (a) and (b) given above show that ${\rm Iso}$ and ${\rm sou}_x$ are equivalences of categories. 
 
\subsection{Representations of a groupoid and associated bundles}


\begin{definition}
A $\mathcal G$-bundle is a bundle $\pi \colon E\to M$ endowed with a Lie groupoid action, that is, a smooth map
$$\mathcal G\,_s\hspace{-1mm}\times E\to E, \quad (g,p)\mapsto g\cdot p,$$
satisfying the following properties:
\begin{itemize}
    \item[(a)] $\mathfrak e_x\cdot p = p$;
    \item[(b)] $(gh)\cdot p = g\cdot(h\cdot p)$;
    \item[(c)] $\pi(g\cdot p) = t(g)$.
\end{itemize}
\end{definition}

Note that, if $E$ is a $\mathcal G$-bundle, then each fiber $E_x$ is a $\mathcal G_{xx}$-manifold. We say that the  $\mathcal G$-bundle $E$ is effective if the fibers $E_x$ are effective $\mathcal G_{xx}$-manifolds. 
If a $\mathcal G$-bundle is not effective, it is possible to factorize the action through a quotient groupoid $\bar{\mathcal G}\rightrightarrows M$ so that $E$ is realized as an effective $\bar{\mathcal G}$-bundle.

\begin{example}
If $P\to M$ is a principal bundle, then $P$ is an ${\rm Iso}(P)$-bundle with the natural action of ${\rm Iso(P)}$ on $P$.
\end{example}

\begin{example}
If $P\to M$ is a principal bundle with structure group $G$ and $V$ is a $G$-manifold then the associated bundle
$E = P\times_G V$ is a ${\rm Iso}(P)$-bundle with the action: 
$${\rm Iso}(P) \,_s\hspace{-1mm}\times E \to E \quad (\sigma,[p,v]) \mapsto [\sigma(p),v].$$
This ${\rm Iso}(P)$-bundle is effective if and only if $V$ is an effective $G$-manifold. 
\end{example}

Reciprocally, let us see that any $\mathcal G$-bundle can be realized as an associated bundle.

\begin{proposition}\label{pr:associated_iso}
Let $E\to M$ be a $\mathcal G$-bundle. For all $x\in M$ we have a isomorphism
$$\mathcal G_{x\bullet}\times_{\mathcal G_{xx}} E_x \xrightarrow{\sim} E$$
of $\mathcal G$-bundles, where the $\mathcal G$-bundle structure of $\mathcal G_{x\bullet}\times_{\mathcal G_{xx}} E_x$ is given by the canonical isomorphism $\mathcal G \simeq {\rm Iso}(\mathcal G_{x\bullet})$.
\end{proposition}

\begin{proof}
The isomorphism of the statement is given explicitly by $[g,e] \mapsto g\cdot e$.
\end{proof}

\subsection{Group bundles as associated bundles}

Let $\mathfrak{G}\to M$ be a locally trivial group bundle. From the local triviality is easy to check that the transitive groupoid ${\rm Iso}_{gr}(\mathfrak{G})\rightrightarrows M$ whose fiber on $(x,y)$ is the set of Lie group isomorphisms from $\mathfrak{G}_x$ to $\mathfrak{G}_y$ admits a canonical structure of Lie groupoid that realizes $\mathfrak{G}$ as an ${\rm Iso}_{gr}(\mathfrak{G})$-bundle. 

\begin{theorem}\label{th:group_associated}
Let $\mathfrak{G}\to M$ be a locally trivial group bundle, $x\in M$, and $G = {\rm Aut}_{gr}(\mathfrak{G}_x)$ the Lie group of Lie group automorphisms of $\mathfrak{G}_x$. Then there is a principal bundle $P\to M$ modeled over $G$ such that:
\begin{itemize}
    \item[(a)] The grupoid ${\rm Iso}(P)$ is canonically isomorphic to ${\rm Iso}_{gr}(\mathfrak{G})$.
    \item[(b)] $H$ is isomorphic to $P\times_G H_x$ as ${\rm Iso}_{gr}(\mathfrak{G})$-bundles. 
\end{itemize}
\end{theorem}

\begin{proof}
Let us take $P = {\rm Iso}_{gr}(\mathfrak G)_{x\bullet}$. Then, statement (a) is a particular case of the isomorphism (b) given in Subsection \ref{eqv_cat}, and statement (b) is a particular case of Proposition \ref{pr:associated_iso}.
\end{proof}

\subsection{Infinitesimally generated $\mathcal G$-connections}

Let us consider $\pi\colon E\to M$ a fiber bundle. There is a canonical infinite dimensional Lie algebroid $\mathfrak{iso}(E)\to M$ such that $\Gamma(\mathfrak{iso}(E)) = \mathfrak X_{\pi}(E)$ is the Lie algebra of $\pi$-projectable vector fields in $E$. We can see it as the union of its finite rank vector subbundles, that are spanned by finite families of $\pi$-projectable vector fields. Thus, the bracket operation in $\Gamma(\mathfrak{iso}(E))$ is the Lie bracket of vector fields and the anchor is the projection onto $TM$. 

\begin{definition}
Let $A\to M$ be a Lie algebroid. 
An infinitesimal action of $A$ in $E$ is a Lie algebroid morphism $\varphi\colon A \to \mathfrak{iso}(E)$.
\end{definition}

Let us consider $\pi\colon E\to M$ a $\mathcal G$-bundle with action $\alpha$. By derivation of $\alpha$ we obtain an infinitesimal action $\alpha'$ of ${\rm Lie}(\mathcal G)$ in $E$,
$$\alpha'\colon {\rm Lie}(\mathcal G) \to \mathfrak{iso}(E),\quad \vec v \mapsto \vec v^*, \quad  \vec v^*(p) = \mathrm{d}\alpha(\mathrm{d} \mathfrak i(\vec v), \vec 0_p)$$ 
Note that if the action of $\mathcal G$ in $E$ is effective then the infinitesimal action $\alpha'$ is injective. \\

Let $\mathcal H$ be a connection in $E$. Note that the $\mathcal H$-horizontal lift $\tilde X$ of any vector field $X$ in $M$ is a $\pi$-projectable vector field in $E$, and therefore a section of $\mathfrak{iso}(E)$.
$${\rm lift}^{\mathcal H}\colon \mathfrak X(M) \to \Gamma(\mathfrak{iso}(E)), \quad X\mapsto {\rm lift}^{\mathcal H}(X) = \tilde X.$$

\begin{definition}
Let $\mathcal H$ be an Ehreshmann connection in a $\mathcal G$-bundle $E$. We say that $\mathcal H$ is an infinitesimally generated $\mathcal G$-connection if it satisfies any of the following equivalent conditions:
\begin{itemize}
    \item[(a)] ${\rm lift}^{\mathcal H}(\mathfrak X(M)) \subset \alpha'(\Gamma({\rm Lie}(\mathcal G)))$.
    \item[(b)] For any piecewise smooth path $\gamma$ with endpoints $x$ and $y$ 
    there is $g\in \mathcal G_{xy}$ such that ${\rm hol}_\gamma^{\mathcal H}\colon v \to g\cdot v$ for all $v\in E_x$. 
\end{itemize}
\end{definition}

Lets us check the equivalence between the two conditions. First, assume that $\gamma$ is an integral curve of a vector field $X$ in $M$ then there is a left invariant vector field $A$ in $\mathcal G$ such that $\alpha'(A) = \tilde X$. Then we have ${\rm hol}_\gamma^{\mathcal H}\colon v \mapsto (\exp(A)\mathfrak e_{\gamma(0)})^{-1}\cdot v$. This proves $(a)\Rightarrow(b)$. Let us now assume that the action of $\mathcal G$ on $E$ is effective, otherwise we take the quotient by the kernel of the action. Let us take a complete vector field $X$  in $M$. Then, by $(b)$ we have that 
${\rm exp}(t\tilde X)|_{E_x}\colon v \mapsto v\cdot g(x,t)$
for certain $g(x,t)\in G$. We define $A = \left.\frac{d}{dt}\right|_{t=0} g(x,t)^{-1}$ which is a section of ${\rm Lie}(\mathcal G)$ such that $\alpha'(A) = \tilde X$.

\begin{theorem}\label{th:GCC}
Let $E\to M$ be a $\mathcal G$-bundle. Fix $x\in M$, $P = \mathcal G_{x\bullet}$, $G = \mathcal G_{xx}$, so that ${\rm Iso}(P)$ is identified with $\mathcal G$ and $E$ is identified with the associated bundle $P\times_G E_x$. Let $\mathcal H$ be an Ehreshmann connection in $E$. The following are equivalent:
\begin{itemize}
    \item[(a)] $\mathcal H$ is an infinitesimally generated $\mathcal G$-connection.
    \item[(b)] There is a $G$-invariant connection $\mathcal D$ in $P$ such that $\mathcal H$ is the associated connection induced by $\mathcal D$ in $E$.
\end{itemize}
Moreover, if $E$ is an effective $\mathcal G$-bundle the connection $\mathcal D$ of statement (b) is unique. 
\end{theorem}

\begin{proof}
Let $\mathcal D$ be a $G$-invariant connection in $P$. Let us consider the identification map,
$$\mathfrak q\colon P\times E_x \to E, \quad (g,v)\mapsto g\cdot v$$
the associated connection to $\mathcal D$ is $\mathfrak q_{*}(\mathcal D,\vec 0)$ which is an infinitesimally generated $\mathcal G$-connection, since its holonomy is given by elements of ${\rm Iso}(P)\simeq \mathcal G$. Note also, that the uniqueness of the holonomy as an element of ${\rm Iso}(P)$ comes from the effectiveness of the action of $G$ in $E_x$. This proves $(b)\Rightarrow(a)$ and the last statement.  

Let us prove $(a)\Rightarrow(b)$.  If the action of $\mathcal G$ in $E$ is not effective we may replace $\mathcal G$ by a quotient $\bar{\mathcal G}$ so that any infinitesimally generated $\mathcal G$-connection is also an infinitesimally generated $\bar{\mathcal G}$-connection. In such case we obtain a principal bundle $\bar P$ which is a quotient of $P$ modeled over a quotient $\bar G$ of $G$. It is well known that any $\bar G$-invariant connection in $P$ lifts to a $G$-invariant connection in $P$ and therefore, we may assume, without loss of generality, that the action is effective. 
Accordingly, we have that for each complete vector field $X$ in $M$ the exponential of $\tilde X$ restricted to a fiber $E_y$ of $E$ is given by an element of $\mathcal G_{y\bullet}$,
$${\exp}(tx)|_{E_y} \colon v \mapsto h(y,t)\cdot v.$$
Let us define for $g\in \mathcal G_{xy}$ the lift to $P$ as $\hat X(g)= \frac{d}{dt} h(y,t)g$ so that this lift operator $X\to \hat X$ defines a $G$-invariant connection $\mathcal D$ in $P$ such that its associated connection
is $\mathcal H$.
\end{proof}

As group connections are infinitesimally generated by its groupoid of isomorphisms, we may apply the above result to them. 
Let us consider a locally trivial group bundle, $\mathfrak{G}\to M$, $x\in M$, $G$, $P$ as in the statement of Theorem \ref{th:group_associated} and $\mathcal H$ a group connection. By direct application of the above theorem we obtain the following result.

\begin{corollary}\label{cor:G_vs_AutG}
There is a unique $G$-invariant connection $\mathcal D$ in $P\to M$ such that $\mathcal H$ is the associated connection of $\mathcal D$.
\end{corollary}

And thus, as a direct consequence of the above and Corollary \ref{cr:AS_associated} we obtain an Ambrose-Singer theorem for group connections. 

\begin{corollary}\label{cor:AS_GC}
Let $\mathcal H$ be a group connection in $\mathfrak{G}\to M$. Then ${\rm Lie}({\rm RHol}_x^{\mathcal H})$ is the Lie subalgebra of $\mathfrak X(\mathfrak{G}_x)$ generated by elements ${\rm hol}_{\gamma*}^{\mathcal H}(R(X,Y))$ where $\gamma$ varies along piece-wise smooth paths with endpoint $x$; $X$, $Y$ vary along vectors with $T_{\gamma(0)}M$ and ${\rm hol}_{\gamma*}^{\mathcal H}\colon \mathfrak X(\mathfrak{G}_{\gamma(0)}) \to \mathfrak X(\mathfrak{G}_x)$ is the Lie algebra morphism induced by ${\rm hol}_{\gamma}^{\bar{\mathcal H}}$.
\end{corollary}

\section{Flat group connections}\label{s_ME}

Let us fix a pointed connected manifold $(M,x_0)$ and a Lie group $G$. Let us consider the class ${\mathcal M}(M,G)$ the Moduli space of isomorphism classes of group connections in group bundles over $M$ with fiber $G$. From corollary \ref{cor:G_vs_AutG} this Moduli space is isomorphic to that of flat ${\rm Aut}(G)$-invariant connections on principal bundles over $M$. This moduli space is well known and isomorphic to the space ${\rm Rep}(\pi_1(M,x_0), {\rm Aut}(G))$.

Let us recall the construction of the space of representations. We consider the left action of ${\rm Aut}(G)$ on ${\rm Hom}(\pi_1(M,x_0), {\rm Aut}(G))$ given by composition with internal automorphisms of ${\rm Aut}(G))$. The space of representations is the quotient 
$${\rm Rep}(\pi_1(M,x_0), {\rm Aut}(G)) = {\rm Hom}(\pi_1(M,x_0), {\rm Aut}(G))/{\rm Aut}(G).$$ 
In what follows, we show how representations of the fundamental group are associated to group connections and vice versa. 

\subsection{From a group connection to a group representation}

Let $\pi:\mathfrak{G}\to M$ be a group bundle with fiber $G$ endowed with a flat group connection $\mathcal{H}$. Let us consider a fixed point $x_0\in M$. We also fix an isomorphism $\mathfrak{G}_{x_0}\simeq G$.  For any given loop $\gamma$ with base point $x_0$ the holonomy ${\rm hol}_\gamma^{\mathcal H}$ depends only of the homotopy class of $\gamma$. Therefore, the holonomy representation factorizes through the homotopy groupoid and we obtain the monodromy representation.

$$
\xymatrix{ \ell(M,x_0,x_0) \ar[d]\ar[rd]^-{{\rm hol}^{\mathcal H}}  \\  \pi_1(M,x_0)  \ar[r]_-{{\rm mon}^{\mathcal H} } &  {\rm Aut}(G) }$$

It is clear that if $(\mathfrak{G},\mathcal H)$ and $(\mathfrak{G}',\mathcal H')$ are isomorphic group bundles with fiber $G$ and flat group connection then $[{\rm mon}^{\mathcal H}]$ does not depends on the choice of the isomorphism $\mathfrak{G}_{x_0}\simeq G$ and $[{\rm mon}^{\mathcal H}] = [{\rm mon}^{\mathcal H'}]$. 
 We have map,
$$\overline{\rm mon} \colon  {\mathcal M}(M,G) \to {\rm Rep}(\pi_1(M,x_0), {\rm Aut}(G)), \quad [\mathcal H]\to 
[{\rm mon}^{\mathcal H}].$$

\subsection{From a representation of to a group connection}

Let $\tilde M$ be the space of all homotopy equivalence classes $[\gamma]$ of paths $\gamma$ in $M$ starting at $x_0$. Then $[\tau]\in \tilde M \to \tau(1)\in M$ is a model for the universal cover of $M$ and we have a $\pi_1(M,x_0)$-principal bundle,
$$\tilde M \times \pi_1(M,x_0)\ni ([\tau],[\gamma]) \mapsto [\tau \gamma]\in \tilde M.$$ If $\rho: \pi_1(M,x_0) \to \mathrm{Aut}(G)$ is a smooth representation of $\pi_1(M,x_0)$ by automorphisms of $G$, we denote by $\Pi:\tilde{M}\times_{\rho} G \to M$ the associated bundle to the representation $\rho$.

\begin{remark}\label{obs:paraoperar}
There are two special observations for this bundle: 
\begin{enumerate}

    \item 
    For each $[\gamma]\in \pi_1(M,x_0)$ and $g\in G$, 
    \[
    [[\gamma],g]=[[x_0], \rho([\gamma])g],\] where $[x_0]$ denotes the class of the constant path;
 
    \item
    Let be $[\tau], [\tau']\in \tilde{M}$ such that $\tau(1)=\tau'(1) = x$. Let $U_x\subset M$ be a simply connected distinguished neighborhood for $x$,  and let $U_{\tau}$ and $U_{\tau'}$ be open neighborhoods for $[\tau]$ and $[\tau']$ respectively and such that both are diffeomorphic to $U_x$.  
        Therefore for each $y\in U_x$ and each $g \in G$ we have that   
    \[
    [[\sigma_y\star\tau'],g]=[[\sigma_y \star \tau \star \gamma^{-1}],g]=[[\sigma_y \star \tau],\rho([\gamma])g].\]
Where $\gamma$ is the loop at $x_0$ given by $\gamma = (\tau')^{-1}\tau$ and $\sigma_y$ is any path in $U_x$ from $x$ to $y$. It follows that the horizontal leaf  $U_{\tau'}\times \{g\} \subset \tilde{M}\times G$ may be identified with the horizontal leaf $U_{\tau}\times \{\rho([\gamma])g\}$ on the associated bundle $\tilde{M}\times_{\rho} G$.
\end{enumerate} 
\end{remark}

The balanced construction $\Pi:\big[[\tau],g\big]\in \tilde{M}\times_{\rho} G \mapsto \tau(1)\in M $ is an group bundle with fiber $G$. Let us consider the projection $\mathfrak q\colon \tilde M\times G \to \tilde M\times_\rho G$. Let us see that the trivial group connection $\mathcal{H}_0=T\tilde{M}\times \{0\}$ on the trivial bundle $\tilde{M}\times  G \to \tilde M$ projects onto a group connection ${\mathfrak q}_*(\mathcal H_0) = \mathcal H_{\rho}$ on the group bundle $M\times_{\rho} G \to M$.
 Given $[\tau], [\tau']\in \tilde{M}$ such that $\tau(1)=\tau'(1) = x_1$, it follows from \ref{obs:paraoperar}, that 
\[
\mathfrak{m}\big(([\tau],g),([\tau'],g')\big)= \mathfrak{m}\big(([\tau],g),([\tau],\rho([\gamma])g')\big)= ([\tau],g\rho([\gamma])g'))
\] establishes a smooth product which turns each fiber of $\tilde{M}\times_{\rho} G$ into a Lie group. 

\medskip

For each $[\gamma]\in \pi_1(M,x_0)$ we have a commutative diagram:
\begin{equation}\label{eq:diagqs}
\vcenter{\xymatrix@C-20pt@R-10pt{\tilde{M}\times G\ar[dr]^{\mathfrak q}\ar[dd]_{R_{[\gamma]^{-1}} \times
L_{[\gamma]}}\\
&\quad \tilde{M}\times_{\rho}G\\
\tilde{M}\times G\ar[ur]_{\mathfrak q}.}}
\end{equation}
It is clear that the differential of
$R_{[\gamma]^{-1}} \times
L_{[\gamma]}$ maps the space   $\mathcal{H}_0([\tau],g)=T_{[\tau]}\tilde{M} \oplus\{0_g\}$ over the space $\mathcal{H}_0([\tau\gamma^{-1}],\rho([\gamma](g))=T_{[R_{\gamma^{-1}}(\tau)]}\tilde{M} \oplus\{0_{L_{[\gamma]}(g)}\}$, for each $[\tau] \in \tilde{M}$ and $g\in G$. Therefore, differentiating the diagram \eqref{eq:diagqs} we obtain:
\[\mathrm{d}_{([\tau],g)}\mathfrak q\big(\mathcal{H}_0([\tau],g\big)=
\mathrm{d}_{(R_{[\gamma]^{-1}}([\tau])],L_{[\gamma]}(g))}\mathfrak q\big(\mathcal{H}_0([\tau\gamma^{-1}],\rho([\gamma](g))\big),\]
it follows that ${\mathcal{H}}_\rho= \mathfrak{q}_*(\mathcal{H}_0)$ defines a group connection on $\tilde{M}\times_{\rho} G$.

\begin{lemma}\label{lema:isomorfismovsmorfismosequivariante}
Let $G$ and $G'$ be two Lie groups and $\rho:\pi_1(M,x)\to \mathrm{Aut}(G)$, $\rho':\pi_1(M,x)\to \mathrm{Aut}(G')$ two group morphisms. There exist a group bundle isomorphism $f: \tilde{M}\times_{\rho}G\to \tilde{M}\times_{\rho'}G'$
such that $f_*(\mathcal H_{\rho}) = \mathcal H_{\rho'}$ if and only if there exists a Lie group isomorphism  $\phi:G\to G'$ which is $\pi_1(M,x)$-equivariant.
\end{lemma}  

\begin{proof} Clearly each connection preserving isomorphism $f$ induces a group homomorphism $f|_x = \phi\colon G\to G'$ which is $\pi(M,x)$-equivariant. On the other hand, let us assume that $\phi:G \to G'$ is a $\pi_1(M,x)$-equivariant Lie group isomorphism. Then the formula:
$$f\colon [[\tau],g]\in M\times_{\rho}G \longmapsto [[\tau],\phi(g')]\in \tilde M\times_{\rho'}G'$$
defines a connection preserving isomorphism of group bundles. 
\end{proof}

Lemma \ref{lema:isomorfismovsmorfismosequivariante} tell us that if $\rho$ and $\rho'$ are in the same class of ${\rm Rep}(\pi_1(M,x_0),{\rm Aut}(G))$ then their balanced constructions are endowed with isomorphic group connections. Therefore,
we have a well defined map:
$$\Psi\colon {\rm Rep}(\pi_1(M,x_0), {\rm Aut}(G)) \to \mathcal M(M,G), \quad [\rho] \to 
[(\tilde M \times_{\rho} G,\mathcal H_{\rho})].$$

\subsection{Equivalence between group representations and group connections} The following lemma ensures that $\overline{\rm mon}\circ \Psi$ is the identity in ${\rm Rep}(\pi_1(M,x_0),{\rm Aut}(G))$.

\begin{lemma} Let us consider $\rho\colon \pi_1(M,x)\to {\rm Aut}(G)$. The monodromy of $\mathcal H_\rho$ is $\rho$.
\end{lemma}
\begin{proof}
Let $\gamma$ be a loop at $x_0$. 
For each $\varepsilon\in [0,1]$ let us consider the path $\gamma_{\varepsilon}\colon [0,1]\to M$ defined by $\gamma_{\varepsilon}(t) = \gamma(\varepsilon t)$. Thus the horizontal lift of $\gamma$ starting on $[[x_0],g]$ is
$$\tilde{\gamma}:[0,1]\to \tilde{M}\times_{\rho} G, t \mapsto [[\gamma_t],g]. 
$$
It follows that
$$[[x_0], {\rm mon}_{x_0}^{{\mathcal H}_\rho}(g)] = \tilde\gamma(1) = [[\gamma],g] =
[[x_0],\rho([\gamma])g],$$
and therefore ${\rm mon}_{x_0}^{{\mathcal H}_\rho}(g) = \rho([\gamma])g$.
\end{proof}

The following lemma ensures that $\psi\circ \overline{\rm mon}$ is the identity in $\mathcal M(M,G_0)$.

\begin{lemma}\label{lema:sobretectividadmapa}
Let $\mathfrak{G}\to M$ be a group bundle endowed with a flat group connection $\mathcal{H}$. Given $x_0\in M$, let $\rho:=\mathrm{mon}^{\mathcal{H}}_{x_0}: \pi_1(M,x_0) \to \mathrm{Aut}(G)$ be its monondromy representation with respect to $\mathcal{H}$. Then
\begin{equation}\label{eq1:isomorfismodefibrados}
\widetilde{\mathrm{mon}}^{\mathcal{H}}_{x_0}:  \big[[\tau],g\big]\in \tilde{M}\times_{\rho} G\longmapsto  \mathrm{mon}^{\mathcal{H}}_{\tau}(g) \in G
\end{equation}is a isomorphism conjugating $\mathcal H_\rho$ and $\mathcal H$. 
\end{lemma}
\begin{proof}
Given $x\in M$, let $\tau:[0,1]\to M$ be a path on $M$ such that $\gamma(0)=x_0$ and $\gamma(1)=x$. Let $\phi_{x}: U_{x}\times \mathfrak{G}_{x}\to \mathfrak{G}|_{U_{x}}$ be a local trivialization induced by $\mathcal{H}$ so that $U_{x}$ is a simply connected neighborhood for the point $x$ and the map:
\[
\psi_x:[\sigma_y\tau,g]\in \tilde{M}\times_{\rho} G|_{U_x} \longmapsto (y,g) \in U_{x}\times G 
\]defines a local trivialization for $\tilde{M}\times_{\rho}G$. Note that both $\phi_x$ and $\psi_x$ map the trivial connections on $U_x\times \mathfrak{G}_x$ and $U_x\times G$ over $\mathcal{H}$ y $\bar{\mathcal{H}}_0$ respectively. The desired result follows from the commutativity of the diagram:
\[
\xymatrix{
\tilde{M}\times_{\rho} G|_{U_x}  \ar[r]^{\sim} \ar[d]_{\psi_{x}} &  \mathfrak{G}|_{U_x}\\ 
U_x\times G\ar[r]_{\cong} &   U_x\times \mathfrak{G}_x \ar[u]_{\phi_x} }
\]
\end{proof}

\medskip

\subsection*{Acknowledgements}
The authors acknowledge the support of their host institutions Universidad Nacional de Colombia, sede Medell\'in, and Universidad de Antioquia. DBS is also grateful for the support of the research group ``Grupo interinstitucional de Investigación en Geometría y Topología''. We thank M. Malakhaltsev, who helped us to clarify some terminology inconsistencies that appeared in the original version of this paper.

\bibliographystyle{plain}
\bibliography{references}
\end{document}